\documentclass[twoside,centertags                  
]{amsart}

\usepackage[cmtip,color,all,curve,arc,poly]{xy}

\usepackage[T1,OT1]{fontenc}
\usepackage[utf8]{inputenc}
\usepackage{hyperref}
\usepackage{amssymb}
\usepackage{leftidx}
\usepackage{mathtools}
\usepackage[usenames,dvipsnames]{xcolor}
\usepackage{enumitem}
\usepackage{tikz}
\usetikzlibrary{matrix,arrows,calc}
\usepackage{tikz-cd}
\usepackage{ifthen}
\usepackage{lscape}

\usepackage{graphicx}
\usepackage{mathrsfs}

\usepackage{float}

\newtheorem{theorem}[equation]{Theorem}
\newtheorem*{theorem*}{Theorem}
\newtheorem{corollary}[equation]{Corollary}
\newtheorem*{corollary*}{Corollary}
\newtheorem{lemma}[equation]{Lemma}
\newtheorem{proposition}[equation]{Proposition}
\theoremstyle{definition}
\newtheorem{definition}[equation]{Definition}
\theoremstyle{remark}
\newtheorem{remark}[equation]{Remark}

\newcommand{\C}[1]{\mathscr{#1}}

\def\op{\mathrm{op}}
\def\env{\mathrm{e}}

\newcommand{\id}[1]{\mathrm{id}_{#1}}

\newcommand{\modules}[1]{\mathrm{Mod}(#1)}

\newcommand{\bimp}[1]{\operatorname{Mod}_p({#1}^\env)}
\newcommand{\bimpst}[1]{\operatorname{\underline{Mod}}_p({#1}^\env)}

\newcommand{\modulesfp}[1]{\mathrm{mod}(#1)}
\newcommand{\proj}[1]{\mathrm{proj}(#1)}

\newcommand{\modulesfpst}[1]{\mathrm{\underline{mod}}(#1)}
\newcommand{\modulesst}[1]{\mathrm{\underline{Mod}}(#1)}
\def\chain{\operatorname{Chain}}

\def\hom{\operatorname{Hom}}
\def\homst{\underline{\operatorname{Hom}}}
\def\rmap{\operatorname{Map}}

\def\gsquare{\operatorname{Sq}}

\def\ext{\operatorname{Ext}}
\def\exttate{\underline{\operatorname{Ext}}}

\def\aut{\operatorname{Aut}}

\def\characteristic{\operatorname{char}}
\def\pic{\operatorname{Pic}}

\def\morita{\operatorname{DGCat}^{\operatorname{Mor}}}

\def\dgalgebras{\operatorname{DGAlg}}
\def\equivalences{\operatorname{DGCat}^{\operatorname{Eq}}}

\newcommand{\etc}[1]{\operatorname{ETC}(#1)}
\newcommand{\ets}[1]{\operatorname{ETS}(#1)}

\def\To{\longrightarrow}

\renewcommand{\ker}{\operatorname{Ker}}
\newcommand{\im}{\operatorname{Im}}

\newcommand{\hc}[2]{\operatorname{C}^{#1}(#2)}
\newcommand{\hz}[2]{\operatorname{Z}^{#1}(#2)}
\newcommand{\hh}[2]{\operatorname{HH}^{#1}(#2)}

\newcommand{\htate}[2]{\underline{\operatorname{HH}}^{#1}(#2)}

\newcommand{\abs}[1]{|#1|}

\def\aut{\operatorname{Aut}}
\def\out{\operatorname{Out}}

\newcounter{rpage}
\newcommand{\rpage}{\value{rpage}} 				
\newcommand{\rescale}{0.3}			
\newcommand{\margen}{0.1}			
\newcommand{\abajo}{4}				
\newcommand{\arriba}{3}
\newcommand{\derecha}{5} 			

\newcounter{trunco}
\newcommand{\trunco}{\value{trunco}}    	

\begin{document}

\title{Enhanced finite triangulated categories}%
\author{Fernando Muro}%
\address{Universidad de Sevilla,
Facultad de Matem\'aticas,
Departamento de \'Algebra,
Avda. Reina Mercedes s/n,
41012 Sevilla, Spain}
\email{fmuro@us.es}
\urladdr{https://personal.us.es/fmuro}

\thanks{The author was partially supported by
the Spanish Ministry of Economy under the grant MTM2016-76453-C2-1-P (AEI/FEDER, UE) and by the Andalusian Ministry of Economy and Knowledge and the Operational Program FEDER 2014--2020 under the grant US-1263032}

\begin{abstract}
	We give a necessary and sufficient condition for the existence of an enhancement of a finite triangulated category. Moreover, we show that enhancements are unique when they exist, up to Morita equivalence.
\end{abstract}

\maketitle
\tableofcontents


\numberwithin{equation}{section}

\section*{Introduction}

Let us fix a perfect ground field $k$. Recall that a field $k$ is perfect if $\characteristic(k)=0$ or if $\characteristic(k)=p$ and any element in $k$ has a $p^{\text{th}}$ root, e.g.~finite and algebraically closed fields, but not function fields in positive characteristic. A (linear) additive category $\C T$ is \emph{finite} if it is idempotent complete, $\dim\C T(X,Y)<\infty$ for any pair of objects $X,Y\in\C T$, and there are finitely many indecomposables up to isomorphism. Such a category is essentially small and satisfies the Krull--Remak--Schmidt property, i.e.~any object decomposes as a finite direct sum of indecomposables with local endomorphism algebra in an essentially unique way. If $X\in\C T$ is a \emph{basic additive generator}, consisting of a direct sum of one indecomposable for each isomorphism class, then $\C T(X,-)\colon\C T\rightarrow\proj{\Lambda}$ is an equivalence onto the category of finitely generated projective right $\Lambda$-modules for $\Lambda=\C T(X,X)$, which is a basic algebra. Conversely, $\proj{\Lambda}$ is finite for any finite-dimensional basic algebra $\Lambda$. Here, \emph{basic} means that, as a right module, $\Lambda$ decomposes as $\Lambda=P_1\oplus\cdots \oplus P_n$ were each $P_i$ is indecomposable and $P_i\ncong P_j$ for $i\neq j$. Any finite-dimensional algebra is Morita equivalent to a basic one, which is unique up to isomorphism.

Finite triangulated categories arise commonly in representation theory, and have been thoroughly studied from that viewpoint \cite{xiao_locally_2005, amiot_structure_2007, krause_report_2012}. If $\C T$ admits a triangulated structure, then $\Lambda$ is a Frobenius algebra. Indeed, Freyd \cite{freyd_stable_1966} showed that projective objects in the category $\modulesfp{\Lambda}$ of finitely presented right $\Lambda$-modules are injective. Therefore, the Baer criterion \cite[Lemma 3.7]{lam_lectures_1999} proves that $\Lambda$ is self-injective, so it is Frobenius since it is basic \cite[Proposition 3.9]{skowronski_frobenius_2011}.

In this paper, we care about enhancements. An \emph{enhanced triangulated category} is just a DG-category $\C A$. Its underlying triangulated category is $D^c(\C A)$, the derived category of compact objects. An \emph{enhanced triangulated structure} on $\C T$ consists of a DG-category $\C A$ and an equivalence $D^c(\C A)\simeq \C T$. We may want to incorporate the \emph{suspension functor} $\Sigma\colon\C T\rightarrow\C T$ to the picture. The pair $(\C T,\Sigma)$ has an enhanced triangulated structure if there is an equivalence $D^c(\C A)\simeq \C T$ which commutes with the suspension functors up to natural isomorphism. Recall that two DG-categories $\C A$ and $\C B$ are \emph{Morita equivalent} if there is a DG-functor $\C A\rightarrow\C B$ such that the induced functor $D^c(\C A)\rightarrow D^c(\C B)$ is an equivalence, or a zig-zag of such DG-functors connecting both.

Our main result is the following theorem, where $\Lambda^\env$ denotes the enveloping algebra of $\Lambda$.

\begin{theorem*}\label{main}
	Let $\C T\simeq\proj{\Lambda}$ be a finite category over a perfect field $k$ with $\Lambda$ a basic Frobenius algebra:
	\begin{enumerate}
		\item $\C T$ has an enhanced triangulated structure if and only if the third syzygy $\Omega^3_{\Lambda^\env}(\Lambda)$ is stably isomorphic to an invertible $\Lambda$-bimodule. 
		
		\item  The possible suspension functors $\Sigma\colon\C T\rightarrow\C T$ such that $(\C T,\Sigma)$ admits an enhanced triangulated structure are $\Sigma\cong-\otimes_\Lambda I$, where $I$ is an invertible $\Lambda$-bimodule whose inverse is stably isomorphic to $\Omega^3_{\Lambda^\env}(\Lambda)$. 
		
		\item If $\Sigma$ is as above, any two enhancements of $(\C T,\Sigma)$ are Morita equivalent.
	\end{enumerate}
\end{theorem*}







Most triangulated categories appearing in the literature are born with an enhancement. This theorem is still interesting for finite triangulated categories for which an enhancement is known, since it shows uniqueness up to Morita equivalence. The stable category $\modulesfpst{A}$ of finite-dimensional modules over a self-injective finite-dimensional algebra $A$ of finite representation type fits in this framework, e.g.~$A=k[x]/(x^n)$. Over an algebraically closed field, these algebras where classified by Riedtmann \cite{riedtmann_algebren_1980, riedtmann_representation-finite_1980, riedtmann_representation-finite_1983}. Over an arbitrary field, there are even more examples and the general picture is yet unknown, see e.g.~\cite{blaszkiewicz_self-injective_2012}. Also the stable category $\underline{\operatorname{MCM}}(A)$ of maximal Cohen--Macaulay modules over a commutative complete local algebra $A$ of finite
Cohen--Macaulay representation type fits here \cite[\S2]{krause_report_2012}, e.g.~$A=k[[x,y]]/(y^2+x^n)$ when $n$ is odd and $k$ is algebraically closed \cite[Proposition 5.11]{yoshino_cohen-macaulay_1990}.

This theorem also applies to the non-standard finite $1$-Calabi-Yau categories defined in \cite{amiot_structure_2007} from deformed preprojective algebras of generalized Dynkin type over an algebraically closed field of characteristic 2 \cite[Theorem 1.3]{bialkowski_deformed_2007}, \cite[Corollary]{bialkowski_deformed_2019}. So far, these categories were only known to be triangulated in the ordinary sense, no enhancements were known (except for those of type $\mathbb L_n$ \cite[Theorems 2 and 3]{bialkowski_deformed_2011}) but (1) shows that an enhancement indeed exists. The simplest of these examples, of Dynkin type $\mathbb D_4$, is the algebra $\Lambda$ obtained as the quotient of the path algebra of the quiver
\begin{center}
	\begin{tikzcd}
		0\arrow[rd, "a_0", shift left=.5ex]&&\\
		&2\arrow[r, "a_2", shift left=.5ex]
		\arrow[lu, "\bar a_0", shift left=.5ex]
		\arrow[ld, "\bar a_1", shift left=.5ex]
		&3\arrow[l, "\bar a_2", shift left=.5ex]\\
		1\arrow[ru, "a_1", shift left=.5ex]&&
	\end{tikzcd}
\end{center}
over an algebraically closed field of characteristic $2$ by the two-sided ideal generated by the following five elements
\begin{align*}
	\bar{a}_0a_0,&&
	\bar{a}_1a_1,&&
	a_2\bar{a}_2,&&
	a_0\bar{a}_0+a_1\bar{a}_1+\bar{a}_2a_2+a_1\bar{a}_1a_0\bar{a}_0,&&
	a_0\bar{a}_0a_1\bar{a}_1+a_1\bar{a}_1a_0\bar{a}_0.
\end{align*}
There are infinitely many known examples of this kind associated to the Dynkin quivers $\mathbb{D}_n$, $n\geq 4$, $\mathbb{E}_7$, and $\mathbb E_8$. There may actually be many more, even in characteristic $\neq 2$, since deformed preprojective algebras associated to these quivers are not yet classified. 

If $\Lambda$ is connected, in the sense that it is not a product of two non-trivial algebras, and not separable, we prove that, if $\C T$ has an enhanced triangulated structure, then there is an essentially unique suspension functor in (2) so, by (3), the enhancement of $\C T$ is unique up to Morita equivalence (Proposition \ref{irreducible}). Deformed preprojective algebras of generalized Dynkin type fit in this family, e.g.~the example above.

If $k$ is algebraically closed, $\Lambda$ is connected, $\C T$ is standard, i.e.~equivalent to the mesh category of its Auslander--Reiten quiver, and we know that an enhancement exists, then (3) was established in \cite{keller_remark_2018}, and was already implicit in \cite{amiot_structure_2007}, where all possible such $\Lambda$ are classified. The simplest of these examples is the algebra of dual numbers $\Lambda=k[x]/(x^2)$, and more generally, the algebra whose representations are periodic chain complexes of any fixed period.

Non-standard examples are not classified, even over algebraically closed fields. Apart from the aforementioned ones, we have the deformed mesh algebras of type $\mathbb B_n$, $n\geq 3$, also in characteristic $2$, see \cite[Example 9.1]{erdmann_periodic_2008}. We do not know any examples in characteristic $\neq 2$. The claimed examples in characteristic $3$ in \cite{bialkowski_nonstandard_2007} are flawed, since the deformations do not satisfy the admissibility condition for the quivers $\mathbb E_n$, $n=6,7,8$, described in \cite{erdmann_periodic_2008}. Nevertheless, there are non-standard Frobenius algebras in arbitrary characteristic \cite[\S4.12]{skowronski_selfinjective_2006}, so there may be many non-standard examples yet to be discovered.

Disconnected examples are also interesting, because we can play with different suspension functors. For instance, $\Lambda=k^n$ satisfies (1) for all $n\geq 1$ because it is separable. In this case $\C T$ is the $n$-fold power of the category $\modulesfp{k}$ of finite-dimensional vector spaces. Moreover, the enveloping algebra of $\Lambda$ is $k^{2n}$, which is semi-simple, hence $\Sigma$ can be the tensor product with an arbitrary invertible $\Lambda$-bimodule, by (2). In this case this is the same as saying that $\Sigma$ can be any permutation of $n$ elements, regarded as an automorphism of $\C T=\modulesfp{k}^n$ in the obvious way, see \cite[Proposition 3.8]{bolla_isomorphisms_1984}. By (3), the pair $(\C T,\Sigma)$ has a unique $k$-linear enhancement up to Morita equivalence. If $\Sigma$ is the cyclic permutation then $(\C T,\Sigma)$ cannot be decomposed as a non-trivial product, i.e.~it is connected as a pair. This is actually the triangulated category considered in \cite[Example]{rizzardo_note_2019} when $n>0$ is even and $k=\ell(x_1,\dots,x_{n+1})$ is a function field on $n+1$ variables (which is perfect in characteristic $0$). In that paper the authors show that this particular $(\C T,\Sigma)$ has non-Morita equivalent enhancements over $\ell$. We deduce that only one of them can be defined over $k$ as per $(3)$.


Hanihara \cite[Theorem 1.2]{hanihara_auslander_2018} proved recently that, under the assumptions of the previous theorem, $\C T$ admits an ordinary triangulated structure if and only if $\Omega^3_{\Lambda^\env}(\Lambda)$ is stably isomorphic to an invertible $\Lambda$-bimodule, see also \cite[Proposition 3.8]{bolla_isomorphisms_1984}. He actually did not use the octahedral axiom at all. Hence, we deduce that $\C T$ admits an ordinary triangulated structure in the sense of Puppe \cite{puppe_formal_1962} if and only if it admits an enhanced triangulated structure. Beware that the underlying ordinary triangulated structure of the latter need not coincide with the former, therefore there may be finite triangulated categories over a perfect field without enhancements (there may even be finite Puppe triangulated categories not satisfying the octahedral axiom!). It would be interesting to know whether this can really happen.

There are many Frobenius algebras which do not satisfy (1) in the previous theorem. Nakayama algebras yield a whole family of examples (Proposition \ref{counter}). Among these, the simplest are $\Lambda=k[x]/(x^n)$, $n>2$.

Theorems on uniqueness of enhancements exist in the literature for triangulated categories coming from algebraic geometry \cite{lunts_uniqueness_2010,canonaco_uniqueness_2018} or algebraic topology \cite{schwede_stable_2001, schwede_uniqueness_2002, schwede_stable_2007}. Nevertheless, enhancements are in general not unique when they exist, see \cite{schlichting_note_2002,kajiura_-enhancements_2013, rizzardo_note_2019} for some algebraic examples and \cite{franke_uniqueness_1996}, \cite[\S2.1]{schwede_stable_2001}, \cite{patchkoria_exotic_2017}, \cite{pst_2018} for topologically flavored examples.


Despite our main result looks very algebraic, some crucial proofs in this paper use homotopy theory. Sets of enhancements appear as connected components of homotopical moduli spaces of enhancements. Some of our papers on the homotopy theory of operads from the last years were developed with this paper's main result in mind. Here we use them together with some Hochschild cohomology computations.

The main theorem is proved at the end, as Corollary \ref{main_theorem}. The claim in the connected but not separable case is Proposition \ref{irreducible}.

\subsection*{Standing assumptions and notation} Throughout this paper, we will work over a fixed ground field $k$, which acts everywhere. In particular, all categories are assumed to be $k$-linear, except for few clear exceptions. The tensor product over $k$ will simply be denoted by $\otimes$, unless otherwise indicated. In our main results, we need $k$ to be perfect. We will make explicit this assumption when needed. It will be a standing assumption in Sections \ref{edge_units_section}, \ref{restricted}, and \ref{vacolap}. Whenever we have a finite category $\C T$, $\Lambda$ will denote the endomorphism algebra of a basic additive generator, as above. We will often work with graded objects, such as graded vector spaces, graded algebras, graded categories, etc. This means $\mathbb Z$-graded, and the Koszul sign rule will be in place. We will also work with bigraded and, occasionally, trigraded objects. Algebras will be regarded as categories with only one object. We will often work with modules over small categories $\C C$, rather than just algebras. We refer to \cite{street_homotopy_1969, auslander_representation_1974} for basic homological algebra in this context. We will denote the Grothendieck abelian category of right $\C C$-modules by $\modules{\C C}$, and $\modulesfp{\C C}$ will be the full subcategory of finitely presented objects. If $\C C$ is graded, then we will actually consider the graded categories of graded modules. For $\C C$ ungraded, modules will also be ungraded, except when explicitly stated otherwise. For a Frobenius algebra $\Lambda$, we will also consider the stable module categories $\modulesst{\Lambda}$ and $\modulesfpst{\Lambda}$, obtained by quotienting out maps factoring through a projective-injective object. Left modules are the same as right modules over the opposite category $\C C^\op$, and bimodules are right modules over the enveloping category $\C C^\env=\C C\otimes\C C^\op$. Further notation and conventions will be introduced the first time we use them.

\subsection*{Acknowledgements} This paper has benefited from valuable conversations, email exchange, and MathOverflow interaction with many colleagues, including Jerzy Białkowski, Marco Farinati, Ramón Flores, Norihiro Hanihara, Dolors Herbera, Bernhard Keller, Henning Krause, Georges Maltsiniotis, Victor Ostrik, Jeremy Rickard, Antonio Rojas, Manuel Saorín, Stefan Schwede, and  Mariano Suárez-Álvarez (in surname alphabetical order). The author wishes to express his gratitude to all of them and to an anonymous referee.

\section{Hochschild cohomology}\label{principio}

The \emph{Hochschild cohomology} of a graded category $\C C$ with coefficients in a $\C C$-bimodule $M$ is 
\[\hh{\star,*}{\C C,M}=\ext^{\star,*}_{\C C^\env}(\C C,M).\]
Here $\star$ is the \emph{Hochschild} or \emph{horizontal degree}, the length of the extension, and $*$ is the \emph{internal} or \emph{vertical degree}, coming from the fact that the category of $\C C$-bimodules is graded. The \emph{total degree} is the sum of both. We will denote by $\abs{x}$ the degree of an element $x$ in a singly graded object, or the total degree of an element in a bigraded object; horizontal and vertical degrees will be denoted by $\abs{x}_h$ and $\abs{x}_v$, respectively. 

Hochschild cohomology can be computed as the cohomology of the \emph{Hochschild cochain complex} $\hc{\star,*}{\C C,M}$, given by
\[\hc{n,*}{\C C,M}=\prod_{X_0,\dots, X_n\in\C C} \hom_{k}^*(\C C(X_1,X_{0})\otimes\cdots\otimes\C C(X_n,X_{n-1}),M(X_n,X_0)),\]
with bidegree $(1,0)$ differential
\begin{multline*}
d(\varphi)(f_1,\dots,f_{n+1})=(-1)^{\abs{\varphi}_v\abs{f_1}}f_1\cdot \varphi(f_2,\dots,f_{n+1})\\+\sum_{i=1}^n(-1)^i \varphi(\dots, f_if_{i+1},\dots)+(-1)^{n+1}\varphi(f_1,\dots,f_n)\cdot f_{n+1}.
\end{multline*}
This complex arises from the \emph{bar projective resolution} $B_\star(\C C)$ of $\C C$ as a $\C C$-bimodule, called standard complex in \cite[\S17]{mitchell_rings_1972}. It is given by
\[B_n(\C C)=\bigoplus_{X_0,\dots, X_n\in\C C} \C C(X_0,-)\otimes\C C(X_1,X_{0})\otimes\cdots\otimes\C C(X_n,X_{n-1})\otimes\C C(-,X_n),\]
with differential
\[d(f_0\otimes\dots\otimes f_{n+1})=\sum_{i=0}^n(-1)^i\cdots\otimes f_if_{i+1}\otimes\cdots\]
and augmentation
\begin{align*}
	\varepsilon\colon B_\star(\C C)&\longrightarrow\C C,&\varepsilon(f_0\otimes f_1)&=f_0f_1.
\end{align*}

We will now describe the algebraic structure on Hochschild cohomology, following the sign conventions in \cite[\S1]{muro_enhanced_2020}. Therein, we define the Hochschild complex of $\C C$ by using the endomorphism operad of the suspension of the graded vector space (with several objects) underlying $\C C$. In this section, we do not suspend (nor we use the language of operads, although it is implicit). Therefore, some formulas here have extra signs coming from the well known operadic suspension as recalled in \cite[Definition 2.4 and Remark 2.5]{muro_cylinders_2016}.

Given two $\C C$-bimodules $M$ and $N$, we have the \emph{cup-product} operation,
\begin{align*}
	\smile\colon\hc{p,q}{\C C,M}\otimes \hc{s,t}{\C C,N}&\To \hc{p+s,q+t}{\C C,M\otimes_{\C C} N},
\end{align*}
which is a map of bigraded complexes
defined by the following formula,
\[(\varphi\smile\psi)(f_1,\dots, f_{p+s})=
(-1)^{t\sum_{i=1}^p\abs{f_i}}
\varphi(f_1,\dots, f_p)\otimes\psi(f_{p+1},\dots,f_{p+s}).\]
This operation comes from the differential graded comonoid structure of $B_\star(\C C)$ in the category of $\C C$-bimodules defined by the comultiplication
\begin{gather*}
	\Delta\colon B_\star(\C C)\longrightarrow B_\star(\C C)\otimes_{\C C}B_\star(\C C),\\
	\Delta(f_0\otimes\cdots\otimes f_{n+1})=\sum_{i=0}^{n}(f_0\otimes\cdots\otimes f_i\otimes 1_{X_i})\otimes(1_{X_i}\otimes f_{i+1}\otimes\cdots\otimes f_{n+1}).
\end{gather*}
The counit is the augmentation $\varepsilon$. 
The induced cup-product on Hochschild cohomology will be denoted in the same way. 
If $M$ is a monoid in the category of $\C C$-bimodules then the cup-product induces a differential bigraded algebra structure on $\hc{\star, *}{\C C,M}$ and a bigraded algebra structure on $\hh{\star,*}{\C C,M}$. The main example is $M=\C C$ itself, the tensor unit in $\C C$-bimodules. The cup-product is bigraded commutative in $\hh{\star,*}{\C C,\C C}$, satisfying the Koszul sign rule with respect to both degrees separately.
If we tweak the differential and the cup product in the following way
\begin{align*}
	d'(\varphi)&=(-1)^{\abs{\varphi}_v}d(\varphi),\\
	\varphi\cdot \psi&=(-1)^{tp}\varphi\smile\psi,
\end{align*}
then the product becomes a map of complexes at the cochain level, and graded commutative in $\hh{\star,*}{\C C,\C C}$, all with respect to the total degree. Actually, this dot product and $d'$ correspond to the cup product and the differential used in \cite{muro_enhanced_2020}. Therefore, the bigraded commutativity of the cup product in $\hh{\star,*}{\C C,\C C}$, as defined here, follows formally from the graded commutativity of the dot product with respect to the total degree, checked in \cite{muro_enhanced_2020}, by purely formal reasons, compare \cite[Remark 2.8]{muro_homotopy_2019}. 

The \emph{Gerstenhaber bracket} endows $\hc{\star,*}{\C C,\C C}$ with a DG-Lie algebra structure of degree $-1$ for the differential $d'$,
\begin{align*}
[-,-]\colon\hc{p,q}{\C C,\C C}\otimes \hc{s,t}{\C C,\C C}&\To \hc{p+s-1,q+t}{\C C,\C C}.
\end{align*}
It is defined as the commutator 
\[[\varphi,\psi]=\varphi\bullet\psi-(-1)^{(p+q-1)(s+t-1)}\psi\bullet\varphi\]
of the \emph{pre-Lie product}
\begin{equation}\label{prelie}
\bullet\colon\hc{p,q}{\C C,\C C}\otimes \hc{s,t}{\C C,\C C}\To \hc{p+s-1,q+t}{\C C,\C C}
\end{equation}
given by
\begin{multline*}
(\varphi\bullet\psi)(f_1,\dots, f_{p+s-1})=\sum_{i=1}^{p}(-1)^{(s-1)(p-i)+t\left(p-1+\sum_{j=1}^{i-1}\abs{f_j}\right)}\varphi(f_1,\dots,f_{i-1},\\\psi(f_{i},\dots, f_{i+s-1}),f_{i+s},\dots,f_{p+s-1}).
\end{multline*}

If $m_2\in\hc{2,0}{\C C,\C C}$ denotes the composition in $\C C$, then $d'=[m_2,-]$. This pre-Lie product corresponds to the brace operation $\varphi\{\psi\}$ in \cite{muro_enhanced_2020}.





The square of the pre-Lie product
\[\gsquare(\varphi)=\varphi\bullet\varphi\]
induces in cohomology an operation, called \emph{Gerstenhaber square},
\[\gsquare\colon\hh{p,q}{\C C,\C C}\To\hh{2p-1,2q}{\C C,\C C}\]
whenever $p+q$ is even or $\characteristic k=2$.

The relations satisfied by the product, the Lie bracket, and the Gerstenhaber square in $\hh{\star,*}{\C C,\C C}$ are:
\begin{align*}
(x\cdot y)\cdot z&=x\cdot (y\cdot z),\\
x\cdot y&=(-1)^{\abs{x}\abs{y}}y\cdot x,\\
[x,y]&=-(-1)^{(\abs{x}-1)(\abs{y}-1)}[y,x],\\
[x,x]&=0, \qquad \abs{x}\text{ odd},\\
[x,[y,z]]&=[[x,y],z]+(-1)^{(\abs{x}-1)(\abs{y}-1)}[y,[x,z]],\\
[x,[x,x]]&=0, \qquad\abs{x}\text{ even},\\
[x,y\cdot z]&=[x,y]\cdot z+(-1)^{(\abs{x}-1)\abs{y}}y\cdot[x,z],\\
\gsquare(x+y)&=\gsquare(x)+\gsquare(y)+[x,y], \qquad \abs{x}, \abs{y}\text{ even or }\characteristic k=2,\\
\gsquare(x\cdot y)&=\gsquare(x)\cdot y^2+x\cdot[x,y]\cdot y+x^2\cdot\gsquare(y),  \qquad \text{idem},\\
[\gsquare(x),y]&=[x,[x,y]],  \qquad \abs{x}\text{ even or }\characteristic k=2.
\end{align*}
We call this algebraic structure a \emph{Gerstenhaber algebra} (usually only the product, the Lie bracket, and a subset of relations are required). The relations show that
\[2\cdot \gsquare(x)=[x,x]\]
whenever the Gerstenhaber square is defined. The Gerstenhaber square and the relations where it appears are redundant if $\characteristic k\neq 2$.

The bivariant functoriality of Hochschild cohomology can be described as in \cite{muro_functoriality_2006}. A graded functor $F\colon \C D\rightarrow \C C$ and a $\C C$-bimodule morphism $\tau\colon N\rightarrow M$ induce morphisms 
on Hochschild cochain complexes that we denote by
\begin{align*}
F^*\colon \hc{\star,\ast}{\C C,M}&\To \hc{\star,\ast}{\C D,M(F,F)},
&F^*(\varphi)(g_1,\dots, g_n)&=\varphi(F(g_1),\dots, F(g_n));\\
\tau_*\colon \hc{\star,\ast}{\C C,N}&\To \hc{\star,\ast+\abs{\tau}}{\C C,M},&
\tau_*(\varphi)(f_1,\dots, f_n)&=\tau(\varphi(f_1,\dots, f_n)).
\end{align*}
The induced morphisms on Hochschild cohomology, denoted in the same way,
\begin{align*}
F^*\colon \hh{\star,\ast}{\C C,M}&\To \hh{\star,\ast}{\C D,M(F,F)},\\
\tau_*\colon \hh{\star,\ast}{\C C,N}&\To \hh{\star,\ast+\abs{\tau}}{\C C,M},
\end{align*}
satisfy $F^*\tau_*=\tau(F,F)_*F^*$. These induced morphisms are compatible with the cup-product in the obvious way, actually at the cochain level. Hochschild cohomology is \emph{Morita invariant}, i.e.~$F^*$ is an isomorphism for any coefficient bimodule provided $F$ induces an equivalence of categories $\modules{\C C}\simeq\modules{\C D}$.

Ungraded objects are graded objects concentrated in degree $0$. An ungraded category $\C C$ has graded and ungraded modules and bimodules. If $\C C$ is an ungraded category and $M$ is a graded $\C C$-bimodule, then graded and ungraded Hochschild cohomology are related by the following obvious formula, \[\hh{p,q}{\C C,M}=\hh{p}{\C C,M^q}.\]

\section{The Gerstenhaber square in the kernel}

In this short section we prove that the Gerstenhaber square operation restricts to the kernel of any morphism induced by a graded functor.

\begin{proposition}\label{kernel}
	Let $F\colon\C D\rightarrow\C C$ be a graded functor. If $x\in \hh{\star,*}{\C C,\C C}$ is in the kernel of $F^*\colon\hh{\star,*}{\C C,\C C}\rightarrow \hh{\star,*}{\C D,\C C(F,F)}$ and $\abs{x}$ is even or $\characteristic k=2$, then $F^*(\gsquare(x))=0$ too.
\end{proposition}

Given a graded functor $F\colon\C D\rightarrow\C C$, we can define an operation
\begin{equation*}
\bullet\colon\hc{p,q}{\C C,\C C}\otimes \hc{s,t}{\C D,\C C(F,F)}\To \hc{p+s-1,q+t}{\C D,\C C(F,F)}
\end{equation*}
by essentially the same formula as the pre-Lie product in \eqref{prelie},
\begin{multline*}
(\varphi\bullet\psi)(g_1,\dots, g_{p+s-1})=\sum_{i=1}^{p}(-1)^{(s-1)(p-i)+t\left(p-1+\sum_{j=1}^{i-1}\abs{g_j}\right)}\varphi(F(g_1),\dots,F(g_{i-1}),\\\psi(g_{i},\dots, g_{i+s-1}),F(g_{i+s}),\dots,F(g_{p+s-1})).
\end{multline*}
We recover \eqref{prelie} for $F$ the identity. This new operation and the pre-Lie product are clearly related by the formula
\[F^*(\varphi\bullet\varphi')=\varphi\bullet F^*(\varphi').\]

The proof of the previous proposition is based on the following formula, which is a slight generalization of \cite[(1.4)]{muro_enhanced_2020}.

\begin{lemma}\label{technical}
	Given $\varphi\in\hc{\star,*}{\C C, \C C}$ and $\psi\in\hc{\star,*}{\C D, \C C(F,F)}$, the following formula holds,
	\begin{equation*}
	F^*(\varphi)\cdot\psi-(-1)^{\abs{\varphi}\abs{\psi}}\psi\cdot F^*(\varphi)=(-1)^{\abs{\varphi}}(d'(\varphi\bullet\psi)-d'(\varphi)\bullet\psi+(-1)^{\abs{\varphi}}\varphi\bullet d'(\psi)).
	\end{equation*}
\end{lemma}

The proof of this lemma is straightforward but rather tedious, hence we omit it.

\begin{proof}[Proof of Proposition \ref{kernel}]
	Let $\varphi\in\hc{\star,*}{\C C, \C C}$ be a representative of $x$, so $d'(\varphi)=0$. Since $F^*(x)=0$, there exists $\psi\in\hc{\star,*}{\C D, \C C(F,F)}$ with $F^*(\varphi)=d'(\psi)$.
	If $x$ has even total degree, so does $\varphi$. Therefore $\abs{\psi}$ is odd and the Leibniz rule for $d'$ together with the formula in Lemma \ref{technical} yield,
	\begin{align*}
		d'(\psi\cdot\psi)&=F^*(\varphi)\cdot\psi-\psi\cdot F^*(\varphi)\\
		&=
		d'(\varphi\bullet\psi)+\varphi\bullet F^*(\varphi).
	\end{align*}
	Hence, the cohomology class of $\varphi\bullet F^*(\varphi)=F^*(\varphi\bullet\varphi)$ is trivial. The same formula proves the result in $\characteristic k=2$, where signs do not matter.
\end{proof}

\section{The Euler derivation}

In this section we recall the Euler class in Hochschild cohomology and we study its interaction with the Gerstenhaber algebra structure.

\begin{definition}
	Given a graded category $\C C$, the  \emph{Euler derivation} $\delta\in\hc{1,0}{\C C,\C C}$ is the cochain defined as
	\[\delta(f)=\abs{f}\cdot f,\qquad f\in \C C(X_1,X_0).\]
\end{definition}

If $\C C$ is ungraded then $\delta=0$, so everything in this section pertains to the graded setting.

\begin{proposition}\label{formulilla}
	Given a Hochschild cochain $\varphi\in\hc{p,q}{\C C,\C C}$, $$[\delta,\varphi]=q\cdot \varphi.$$
\end{proposition}

\begin{proof}
	The result is a consequence of the following equations,
	\begin{align*}
	[\delta,\varphi]&=\delta\bullet \varphi-\varphi\bullet \delta,\\
	(\delta\bullet \varphi)(f_1,\dots,f_p)&=\delta(\varphi(f_1,\dots,f_p))\\
	&=\left(q+\sum_{i=1}^{p}\abs{f_i}\right)\cdot\varphi(f_1,\dots,f_p)\\
	\varphi\bullet \delta&=\sum_{i=1}^p\varphi(f_1,\dots,f_{i-1},\delta(f_i),f_{i+1},\dots, f_p)\\
	&=\sum_{i=1}^p\abs{f_i}\varphi(f_1,\dots, f_p).
	\end{align*}
\end{proof}

\begin{corollary}
	The Euler derivation $\delta$ is a Hochschild cocycle.
\end{corollary}

\begin{proof}
	Simply note that $[m_2,\delta]=-[\delta,m_2]=0\cdot m_2=0$ since $m_2\in\hc{2,0}{\C C,\C C}$.
\end{proof}

The \emph{Euler class} is the Hochschild cohomology class of the Euler derivation 
$$\{\delta\}\in\hh{1,0}{\C C,\C C}.$$

\begin{corollary}\label{euler_bracket}
	Given $x\in \hh{*,q}{\C C,\C C}$,
	$[\{\delta\},x]=q\cdot x$.
\end{corollary}

\begin{proposition}\label{euler_gsquare}
	If $\characteristic k=2$, then $\gsquare(\delta)=\delta\bullet\delta=\delta$, hence $\gsquare(\{\delta\})=\{\delta\}$.
\end{proposition}

It suffices to notice that any integer is congruent to its square modulo $2$.

\begin{proposition}\label{nulo}
	Given $y\in\hh{n,-1}{\C C,\C C}$ with $n$ odd, and $x\in \hh{p,q}{\C C,\C C}$, then
	\[y\cdot x=[y,\{\delta\}\cdot x]+\{\delta\}\cdot[y, x].\]
\end{proposition}

\begin{proof}
	Using Proposition \ref{formulilla} and the laws of a Gerstenhaber algebra,
	\begin{align*}
	[y,\{\delta\}]&=-[\{\delta\},y]&[y,\{\delta\}\cdot x]&=[y,\{\delta\}]\cdot x-\{\delta\}\cdot[y, x]\\
	&=y,&	&=y\cdot x-\{\delta\}\cdot[y, x].
	\end{align*}
\end{proof}

\begin{remark}\label{nulo2}
	The formula in the previous proposition shows that  $\{\delta\}\cdot-$  is a chain null-homotopy for $y\cdot -$, if we think of $[y,-]$ as a differential. If $\gsquare(y)=0$, then $[y,-]$ is a differential and multiplication by $y$ is a chain map since
	\begin{align*}
	[y,[y,x]]&=[\gsquare(y),x]=0,\\
	[y,y\cdot x]&=[y,y]\cdot x+y\cdot[y,x]=2\cdot\gsquare(y)\cdot x+y\cdot[y,x]=y\cdot[y,x].
	\end{align*}
	This observation is crucial for the proof of Theorem \ref{classification}.
\end{remark}

\begin{proposition}\label{vaa0}
	The square of the Euler class vanishes,\[\{\delta\}^2=0\in\hh{2,0}{\C C,\C C}.\]
\end{proposition}

\begin{proof}
	Consider the cochain $\beta\in\hc{1,0}{\C C,\C C}$ defined by
	\[\beta(f)=\frac{\abs{f}(1-\abs{f})}{2}\cdot f,\qquad f\in\C C(X_1,X_0).\]
	On the one hand,
	\begin{align*}
	(\delta\cdot\delta)(f_1,f_2)={}&\delta(f_1)\cdot \delta(f_2)\\
	={}&\abs{f_1}\abs{f_2}f_1\cdot f_2.
	\end{align*}
	On the other hand,
	\begin{multline*}
	d(\beta)(f_1,f_2)=f_1\cdot\beta(f_2)-\beta(f_1\cdot f_2)+\beta(f_1)\cdot f_2\\
	=\left(\frac{\abs{f_1}(1-\abs{f_1})}{2}-\frac{(\abs{f_1}+\abs{f_2})(1-(\abs{f_1}+\abs{f_2}))}{2}
	+\frac{\abs{f_2}(1-\abs{f_2})}{2}\right)f_1\cdot f_2\\
	=\abs{f_1}\abs{f_2}f_1\cdot f_2.
	\end{multline*}
	Hence, $d(\beta)=\delta^2$.
\end{proof}

\begin{proposition}\label{euler_sub_lie}
	Given $x\in\hh{p,q}{\C C,\C C}$ and $y\in\hh{s,t}{\C C,\C C}$,
	\[[\{\delta\}\cdot x,\{\delta\}\cdot y]=(t-q)\cdot \{\delta\}\cdot x\cdot y.\]
	In particular, if $\characteristic k\neq 2$ and $p+q$ is odd then
	\[\gsquare(\{\delta\}\cdot x)=0.\]
\end{proposition}

\begin{proof}
	The following equations, which are consequences of the laws of a Gerstenhaber algebra and Propositions \ref{formulilla} and \ref{vaa0}, prove the claim,
	\begin{align*}
		[\{\delta\}\cdot x,\{\delta\}\cdot y]={}&[\{\delta\}\cdot x,\{\delta\}]\cdot y+(-1)^{p+q}\{\delta\}\cdot[\{\delta\}\cdot x,y]\\={}&-q\{\delta\}\cdot x\cdot y
		+(-1)^{p+q}\{\delta\}\cdot\{\delta\}\cdot [x,y]\\&
		+(-1)^{p+q+(p+q)(s+t-1)}\{\delta\}\cdot[\{\delta\},y]\cdot x\\={}&-q\{\delta\}\cdot x\cdot y
		+(-1)^{(p+q)(s+t)}t\{\delta\}\cdot y\cdot x\\={}&-q\{\delta\}\cdot x\cdot y
		+t\{\delta\}\cdot x\cdot y.
	\end{align*}
\end{proof}

\begin{proposition}\label{euler_gsquare_2}
	If $\characteristic k=2$ and $x\in\hh{\star,q}{\C C,\C C}$ then \[\gsquare(\{\delta\}\cdot x)=(q+1)\cdot \{\delta\}\cdot x^2.\]
\end{proposition}

\begin{proof}
	Using the Gerstenhaber algebra relations together with Propositions \ref{euler_gsquare} and \ref{vaa0} and Corollary \ref{euler_bracket} we obtain
	\begin{align*}
	\gsquare(\{\delta\}\cdot x)&=\gsquare(\{\delta\})\cdot x^2+\{\delta\}\cdot [\{\delta\},x]\cdot x+\{\delta\}^2\cdot\gsquare(x)\\
	&=\{\delta\}\cdot x^2+\{\delta\}\cdot (q\cdot x)\cdot x\\
	&=(q+1)\cdot \{\delta\}\cdot x^2.
	\end{align*}
\end{proof}

\section{Hochschild cohomology of weakly stable graded categories}\label{weakly_stable}

Given an ungraded category $\C T$ equipped with an automorphism $\Sigma\colon\C T\rightarrow\C T$, we can form a graded category $\C T_\Sigma$ with the same object set, morphism objects given by
\[\C T_\Sigma^n(X,Y)=\C T(X,\Sigma^n Y),\qquad n\in\mathbb Z,\]
and composition
\begin{align*}
\C T_{\Sigma}^p(Y,Z)\otimes \C T_{\Sigma}^q(X,Y)&\To\C T_\Sigma^{p+q}(X,Z),\\ 
f\otimes g&\;\mapsto\; (\Sigma^qf)g.
\end{align*}
Here, on the right, we have a composition in $\C T$. The degree $0$ part of $\C T_\Sigma$ is precisely $\C T$, and we denote the inclusion by $i\colon\C T\subset\C T_\Sigma$. This can also be done if $\Sigma$ is just a self equivalence, after choosing an adjoint inverse.

We can extend $\Sigma$ to an automorphism $\Sigma \colon\C T_\Sigma\rightarrow\C T_\Sigma$, defined as in $\C T$ on objects, and on morphisms as $(-1)^n\Sigma$ in each degree $n\in\mathbb Z$. In this way, the graded $\Sigma$ is equipped with a natural isomorphism 
\begin{equation}\label{imath}
\imath_X\colon X\cong \Sigma X
\end{equation} of degree $-1$ given by the identity in $X$. The sign in the extension of $\Sigma$ is necessary for graded naturality, because of Koszul's sign rule. Graded categories equivalent to some $\C T_\Sigma$ are called \emph{weakly stable} \cite{street_homotopy_1969}. They are characterized by the fact that shifts of representable functors are representable, or equivalently, each object has an isomorphism of any given degree, i.e.~for any object $X$ and any $n\in\mathbb{Z}$ there exists an object $Y$ and an isomorphism $X\rightarrow Y$ of degree $n$.

In the special case of an ungraded algebra $\Lambda$ equipped with an automorphism $\sigma\colon\Lambda\rightarrow\Lambda$, the graded algebra given by the previous construction will rather be denoted by $\Lambda(\sigma)$, so as not to confuse it with a twisted bimodule. It can be described as
\[\Lambda(\sigma)=\frac{\Lambda\langle\imath^{\pm1}\rangle}{(\imath x-\sigma(x)\imath)_{x\in\Lambda}},\qquad\abs{\imath}=-1.\]
The extension of $\sigma$ to $\Lambda(\sigma)$ is given by $\sigma(\imath)=-\imath$. The degree $0$ part of $\Lambda(\sigma)$ is $\Lambda$, and the degree $n$ part is isomorphic to the twisted $\Lambda$-bimodule ${}_{\sigma^n}\Lambda_{1}$. Recall that, given two automorphisms $\phi,\psi\colon\Lambda\rightarrow\Lambda$ and a $\Lambda$-bimodule $M$, the \emph{twisted $\Lambda$-bimodule} ${}_\phi M_\psi$ has underlying $k$-vector space $M$ and bimodule structure $\cdot$ given by $a\cdot x\cdot b=\phi(a)x\psi(b)$. Here, $a,b\in\Lambda$, $x\in M$, and the product on the right of the equation is given by the $\Lambda$-bimodule structure of $M$. Up to isomorphism, $\Lambda(\sigma)$ only depends on the class of  $\sigma$ in the outer automorphism group $\out(\Lambda)$ of $\Lambda$. It is worth to notice that the twisted $\Lambda$-bimodule ${}_\phi\Lambda_\psi$ is isomorphic to $\Lambda$ both as a left and as a right $\Lambda$-module, but in general not as a bimodule.

\begin{remark}\label{signos}
	The isomorphism between the degree $n$ part of $\Lambda(\sigma)$ and ${}_{\sigma^n}\Lambda_{1}$ is given by the generator $\imath^{-n}$. Therefore, $\sigma\colon \Lambda(\sigma)\rightarrow\Lambda(\sigma)$ in degree $n$ corresponds to the $\Lambda$-bimodule isomorphism $(-1)^n\sigma\colon {}_{\sigma^n}\Lambda_{1}\rightarrow {}_{\sigma^{n+1}}\Lambda_{\sigma}$.
\end{remark}

If $\C T$ is finite and $\Lambda$ is the endomorphism algebra of a basic additive generator $X\in\C T$, then $\Lambda(\sigma)$ is the endomorphism algebra of $X$ in $\C T_\Sigma$ for $\sigma$ an automorphism induced by $\Sigma$. Here we use that $\Sigma$ preserves the basic additive generator up to isomorphism since $\Sigma$ is an equivalence and the basic additive generator is a direct sum of one representative for each isomorphism class of indecomposables in $\C T$.   The class of the automorphism $\sigma$ in the outer automorphism group is well defined.

In the following statement we use shifted modules over a graded algebra. If $A^*$ is a graded algebra and $M^*$ is a left $A^*$-module, then the shifted object $M^{*-1}$ is a left $A^*$-module, and the product $a m$ in $M^{*-1}$ is defined as $(-1)^{\abs{a}}a m$ in $M^*$. If $M^*$ is a right $A^*$-module, then there is no sign twisting in the definition of the right action of $A^*$ on $M^{*-1}$.

\begin{proposition}\label{graded_ungraded_long_exact_sequence}
	Let $i\colon\C T\subset\C T_{\Sigma}$ be the inclusion of the degree $0$ part. Then there is a long exact sequence
	\begin{center}
		\begin{tikzcd}[row sep=2mm]
			\cdots\arrow[r]&
			\hh{n,*}{\C T_{\Sigma},\C T_{\Sigma}}\arrow[r, "i^*"]\arrow[d, phantom, ""{coordinate, name=Z}]&
			\hh{n,*}{\C T,\C T_{\Sigma}}\arrow[r, "\operatorname{id}-\Sigma_*^{-1}\Sigma^*"]&[20pt]
			\hh{n,*}{\C T,\C T_{\Sigma}}\arrow[dll,"\partial",
			rounded corners,
			to path={ -- ([xshift=2ex]\tikztostart.east)
				|- (Z) [near end]\tikztonodes
				-| ([xshift=-2ex]\tikztotarget.west)
				-- (\tikztotarget)}]\\
			&\hh{n+1,*}{\C T_{\Sigma},\C T_{\Sigma}}\arrow[r]&
			\cdots
		\end{tikzcd}
	\end{center}
	where $i^*$ and $\Sigma^{-1}_*\Sigma^*$	are graded algebra morphisms, the map $\partial\colon \hh{\star-1,*}{\C T,\C T_{\Sigma}}\rightarrow\hh{\star,*}{\C T_{\Sigma},\C T_{\Sigma}}$
	is an $\hh{\star,*}{\C T_{\Sigma},\C T_{\Sigma}}$-bimodule morphism whose image is a square-zero ideal, and the composite $\partial i^*$ is left multiplication by $\{\delta\}$.
\end{proposition}

\begin{proof}
	The maps $i^*$, $\Sigma_*$, and $\Sigma^*$ are graded algebra morphisms by functoriality, actually differential graded algebra morphisms on cochains. The long exact sequence was constructed in \cite[Proposition 2.3]{muro_first_2019}, where we show that
		\begin{center}
		\begin{tikzcd}[column sep=15mm]
		\hc{\star,*}{\C T_{\Sigma},\C T_{\Sigma}}\arrow[r, "i^*"]&
		\hc{\star,*}{\C T,\C T_{\Sigma}}\arrow[r, "\operatorname{id}-\Sigma_*^{-1}\Sigma^*"]&
		\hc{\star,*}{\C T,\C T_{\Sigma}}
		\end{tikzcd}
	\end{center}
	can be extended to an exact triangle of complexes. 
	An explicit null-homotopy for $(\operatorname{id}-\Sigma_*^{-1}\Sigma^*)i^*$ is given by
	\[h\colon \hc{n,*}{\C T_{\Sigma},\C T_{\Sigma}}\longrightarrow \hc{n-1,*}{\C T,\C T_{\Sigma}},\]
	\begin{align*}
	h(\varphi)(f_1,\dots,f_{n-1})&=\sum_{i=0}^{n-1}(-1)^i\imath_{X_0}^{-1}\varphi(\Sigma f_1,\dots,\Sigma f_i,\imath_{X_i},f_{i+1},\dots, f_{n-1}).
	\end{align*}
	The equation 
	\[(\id{}-\Sigma_*^{-1}\Sigma^*)i^*=d'h+hd'\]
	follows easily from the definitions and the naturality of $\imath$. 
	
	The standard exact triangle completion of $\operatorname{id}-\Sigma_*^{-1}\Sigma^*$, that we will denote by $\hc{\star,*}{\C T,\Sigma}$, is
	\[\hc{\star,*}{\C T,\C T_{\Sigma}}\oplus\hc{\star-1,*}{\C T,\C T_{\Sigma}}\]
	endowed with the differential 
	\[
	\left(\begin{array}{cc}
	d'&0\\
	\operatorname{id}-\Sigma_*^{-1}\Sigma^*&-d'
	\end{array}\right).
	\]
	This is the desuspension of the mapping cone of $\operatorname{id}-\Sigma_*^{-1}\Sigma^*$. 
	The previous null-homotopy defines an explicit quasi-isomorphism
	\begin{equation}\label{dga-quasi-iso}
	\begin{split}
	\hc{\star,*}{\C T_\Sigma,\C T_\Sigma}&\longrightarrow \hc{\star,*}{\C T,\Sigma},\\
	\varphi&\;\mapsto\;(i^*(\varphi),h(\varphi)).
	\end{split}
	\end{equation}
	Indeed, $\hc{\star,*}{\C T,\Sigma}$ can be obtained by applying $\hom^*_{\C T_\Sigma^\env}(-,\C T_\Sigma)$ to the resolution of $\C T_\Sigma$ constructed in the proof of \cite[Proposition 2.3]{muro_first_2019}, in the same way as $\hc{\star,*}{\C T_\Sigma,\C T_\Sigma}$ is built from the bar resolution, and \eqref{dga-quasi-iso} is defined by a map of resolutions.

	The complex $\hc{\star,*}{\C T,\Sigma}$ has a semi-direct product differential graded algebra structure given by
	\[(\varphi,\varphi')\cdot(\psi,\psi')=(\varphi\cdot\psi,\varphi'\cdot\psi+(-1)^{\abs{\varphi}}\Sigma^{-1}_*\Sigma^*(\varphi)\cdot\psi').\]
	This is indeed the very definition of the semi-direct product of the graded algebra $\hc{\star,*}{\C T,\C T_{\Sigma}}$ and the twisted and then shifted $\hc{\star,*}{\C T,\C T_{\Sigma}}$-bimodule \[{}_{\Sigma_*^{-1}\Sigma^*}\hc{\star-1,*}{\C T,\C T_{\Sigma}}_1.\] It is straightforward to check that the differential of $\hc{\star,*}{\C T,\Sigma}$ satisfies the Leibniz rule with respect to this product. 

	The quasi-isomorphism \eqref{dga-quasi-iso} is also a differential graded algebra map since the null-homotopy $h$ satisfies the following kind of Leibniz rule twisted by $\Sigma^{-1}_*\Sigma^*$,
	\begin{align*}
		h(\varphi\cdot\psi)&=h(\varphi)\cdot i^*\psi+(-1)^{\abs{\varphi}}\Sigma^{-1}_*\Sigma^*i^*(\varphi)\cdot h(\psi).
	\end{align*}
	This is a tedious but straightforward checking.


	


	The long exact sequence in the statement is therefore also defined by the following short exact sequence of complexes
	\[\hc{\star-1,*}{\C T,\C T_{\Sigma}}\hookrightarrow\hc{\star,*}{\C T,\Sigma}\twoheadrightarrow\hc{\star,*}{\C T,\C T_{\Sigma}}.\]
	Here, the map on the right is the projection onto the first coordinate, which is a differential graded algebra morphism. The map on the left is the inclusion of the second factor, which is a $\hc{\star,*}{\C T,\Sigma}$-bimodule morphism if we first regard the source as 
	the twisted and then shifted $\hc{\star,*}{\C T,\C T_{\Sigma}}$-bimodule ${}_{\Sigma_*^{-1}\Sigma^*}\hc{\star-1,*}{\C T,\C T_{\Sigma}}_1$ and then we pull it back along the second map. The cohomology of ${}_{\Sigma_*^{-1}\Sigma^*}\hc{\star-1,*}{\C T,\C T_{\Sigma}}_1$ is the $\hh{\star,*}{\C T,\C T_{\Sigma}}$-bimodule ${}_{\Sigma_*^{-1}\Sigma^*}\hh{\star-1,*}{\C T,\C T_{\Sigma}}_1$. If we restrict coefficients to 
	$\hh{\star,*}{\C T_{\Sigma},\C T_{\Sigma}}$ along $i^*$, then ${}_{\Sigma_*^{-1}\Sigma^*}\hh{\star-1,*}{\C T,\C T_{\Sigma}}_1$ coincides with the bimodule $\hh{\star-1,*}{\C T,\C T_{\Sigma}}$, with no twisting, because $i^*$ maps $\hh{\star,*}{\C T_{\Sigma},\C T_{\Sigma}}$ to the equalizer of $\Sigma_*^{-1}\Sigma^*$ and the identity in $\hh{\star,*}{\C T,\C T_{\Sigma}}$. Therefore, $\partial$ is an $\hh{\star,*}{\C T_{\Sigma},\C T_{\Sigma}}$-bimodule morphism. Moreover, the second direct summand in $\hc{\star,*}{\C T,\Sigma}$ is a square-zero two-sided ideal by definition, hence the image of $\partial$ is a square-zero ideal in $\hh{\star,*}{\C T_{\Sigma},\C T_{\Sigma}}$.
	
	In order to conclude this proof, we must check that the map in cohomology induced by the endomorphism of $\hc{\star,*}{\C T,\Sigma}$ given by $(\varphi,\psi)\mapsto(0,-\varphi)$ coincides with left multiplication by $\{\delta\}$. The image of the Euler derivation in $\hc{\star,*}{\C T,\Sigma}$ along \eqref{dga-quasi-iso} is $(0,-1)$ since
	\begin{align*}
		i^*(\delta)&=0,\\
		h(\delta)(X)&=\imath_X^{-1}\cdot \delta(\imath_X)\\&=|\imath_X|\imath_X^{-1}\cdot \imath_X\\&=-1,
	\end{align*}
	and $(0,-1)\cdot(\varphi,\psi)=(0,-\varphi)$, hence we are done.
\end{proof}

\section{Hochschild--Tate cohomology}\label{hochschild_tate}

In this section we recall stable Hochschild cohomology, as introduced in \cite{eu_calabi-yau_2009}. We emphasize some explicit constructions and perform computations which will be useful for our purposes.

The \emph{stable module category} $\modulesst{\Lambda}$ of a Frobenius algebra $\Lambda$ is the quotient of $\modules{\Lambda}$ by the ideal of maps which factor through a projective-injective object. The \emph{syzygy functor} $\Omega\colon \modulesst{\Lambda}\rightarrow\modulesst{\Lambda}$, which is an equivalence, is defined by the choice of short exact sequences
\[\Omega(M)\hookrightarrow P\twoheadrightarrow M\]
with projective-injective middle term. Moreover, $\modulesst{\Lambda}$ is triangulated with suspension functor $\Omega^{-1}$, the \emph{cosyzygy functor}, an inverse equivalence of $\Omega$ which is similarly defined by the choice of short exact sequences
\[M\hookrightarrow P'\twoheadrightarrow \Omega^{-1}(M)\]
with projective-injective middle term. Graded morphisms in $\modulesst{\Lambda}_{\Omega^{-1}}	$ are called \emph{Tate} or \emph{stable $\ext$ functors}, $n\in\mathbb Z$,
\begin{multline*}
\exttate^n_\Lambda(M,N)=\modulesst{\Lambda}_{\Omega^{-1}}^n(M,N)=\homst_{\Lambda}(M,\Omega^{-n}(N))\\
\cong\homst_{\Lambda}(\Omega^n(M),N)\cong H^n\hom_{\Lambda}(P_\star,N).
\end{multline*}
Here, $\homst_{\Lambda}$ denotes morphism vector spaces in $\modulesst{\Lambda}$ and $P_\star$ is a \emph{complete} or \emph{two-sided resolution} of $M$. The obvious projection $P_\star\twoheadrightarrow P_{\geq 0}$ defines natural \emph{comparison maps}, $n\geq 0$,
\[\ext_\Lambda^n(M,N)\longrightarrow\exttate^n_\Lambda(M,N)\]
which are isomorphisms for $n>0$ and onto for $n=0$. Composition in $\modulesst{\Lambda}_{\Omega^{-1}}$ extends the Yoneda product.

The \emph{Hochschild--Tate} or \emph{stable Hochschild cohomology} of $\Lambda$ with coefficients in a $\Lambda$-bimodule $M$ is defined as
\[\htate{\star}{\Lambda,M}=\exttate_{\Lambda^\env}^\star(\Lambda,M).\]
In particular, we have \emph{comparison maps}
\[\hh{\star}{\Lambda,M}\longrightarrow\htate{\star}{\Lambda,M}\]
which are isomorphisms for $n>0$ and onto for $n=0$. Hochschild--Tate cohomology, as a functor on the coefficients, is functorial in $\modulesst{\Lambda^\env}$.

The category $\modules{\Lambda^\env}$ is monoidal for the tensor product $\otimes_{\Lambda}$ and $\Lambda^\env$ is Frobenius \cite[Lemmas 3.1 and 3.2]{bergh_tate-hochschild_2013}, but $\modulesst{\Lambda^\env}$ does not inherit the monoidal structure because the functor $M\otimes_{\Lambda}-$ is not exact and it does not preserve projective-injective objects in general. Nevertheless, the full subcategory $\bimp{\Lambda}\subset\modules{\Lambda^\env}$ spanned by those $\Lambda$-bimodules which are left and right projective is a Frobenius exact category containing all projective-injective $\Lambda$-bimodules. Moreover, the tensor product $\otimes_{\Lambda}$ restricts to $\bimp{\Lambda}$ and the functors $M\otimes_{\Lambda}-$ and $-\otimes_{\Lambda}M$ are exact and preserve projective-injective objects for $M$ in $\bimp{\Lambda}$. Therefore, the stable category $\bimpst{\Lambda}$, which is a full triangulated subcategory of $\modulesst{\Lambda^\env}$, inherits the monoidal structure $\otimes_{\Lambda}$. 

From now on, in this and in the following two sections, we exceptionally use the symbol $\otimes$ to denote $\otimes_{\Lambda}$. The $k$-linear tensor product will therefore be denoted by $\otimes_k$ here. 

The tensor product functor in $\bimpst{\Lambda}$ is triangulated in both variables, so there are natural isomorphisms in $\bimpst{\Lambda}$
\begin{equation}\label{isos}
	\Omega(M)\otimes  N\cong \Omega(M\otimes  N)\cong M\otimes  \Omega(N).
\end{equation}
These isomorphisms satisfy coherence conditions that can be easily deduced from the explicit construction below. 
Moreover, the square
\begin{equation}\label{anti}
	\begin{tikzcd}
		\Omega(M)\otimes \Omega(N)\ar[r]\ar[d] & \Omega(M\otimes \Omega(N))\ar[d]\\
		\Omega(\Omega(M)\otimes N) \ar[r] & \Omega^2(M\otimes N) 
	\end{tikzcd}
\end{equation}
commutes up to a $-1$ sign. This follows from Lemma \ref{sign} below. In particular, $\bimpst{\Lambda}$ fits in the general framework described in \cite{suarez-alvarez_hilton-heckmann_2004}.

In order to fix a syzygy functor $\Omega$ in $\bimpst{\Lambda}$, we first define $\Omega(\Lambda)$ via any short exact sequence in $\bimp{\Lambda}$
\[\Omega(\Lambda)\stackrel{j}\hookrightarrow P\stackrel{p}\twoheadrightarrow\Lambda\]
with $P$ projective-injective, and then $\Omega (M)$ via
\[\Omega(\Lambda)\otimes  M\hookrightarrow P\otimes  M\twoheadrightarrow\Lambda\otimes  M\cong M.\]
For $M=\Lambda$, both definitions are equivalent via the canonical isomorphism $\Omega(\Lambda)\otimes\Lambda\cong\Omega(\Lambda)$. With this choice, $\Omega=\Omega(\Lambda)\otimes -$, the first isomorphism in \eqref{isos} is the identity, and the second one is given by a coherent natural isomorphism in $\bimpst{\Lambda}$
\begin{equation}\label{zeta}
\zeta_M\colon \Omega(\Lambda)\otimes  M\cong M\otimes \Omega(\Lambda)
\end{equation}
defined by any choice of map of extensions in $\bimp{\Lambda}$ of the following form
\begin{center}
	\begin{tikzcd}
	\Omega(\Lambda)\otimes  M\ar[r, hook, "j\otimes 1"] \ar[d] & P\otimes  M \ar[r, two heads, "p\otimes 1"] \ar[d] & \Lambda\otimes  M\cong M\ar[d, equal, xshift=22]\\
	M\otimes  \Omega(\Lambda)\ar[r, hook, "1\otimes j"] & M\otimes  P \ar[r, two heads, "1\otimes p"] & M\otimes  \Lambda\cong M
	\end{tikzcd}
\end{center}
Coherence here means that the following diagrams commute
\begin{center}
	\begin{tikzcd}[column sep = -20]
		&\Omega(\Lambda)\otimes M\otimes N \ar[rd, "\zeta_{M\otimes N}"]\ar[ld, "\zeta_{M}\otimes 1"']&\\
		M\otimes \Omega(\Lambda)\otimes N\ar[rr, "1\otimes \zeta_{N}"]&&M\otimes N\otimes\Omega(\Lambda)
	\end{tikzcd}
	\begin{tikzcd}[column sep = -0]
		\Omega(\Lambda)\otimes\Lambda \ar[rr, "\zeta_{\Lambda}"] \ar[rd, "\cong"']&&\Lambda\otimes\Omega(\Lambda)\ar[ld, "\cong"]\\
		&\Omega(\Lambda)&
	\end{tikzcd}
\end{center}
Here, and elsewhere, we omit associativity constraints.

\begin{lemma}\label{sign}
	The automorphism $\zeta_{\Omega(\Lambda)}\colon \Omega(\Lambda)\otimes \Omega(\Lambda)\cong\Omega(\Lambda)\otimes \Omega(\Lambda)$ in $\bimpst{\Lambda}$ is $-1$. In particular, $\zeta_{\Omega(\Lambda)}^{-1}=\zeta_{\Omega(\Lambda)}$.
\end{lemma}

\begin{proof}
	This result does not depend on the short exact sequence chosen to define $\Omega(\Lambda)$. Therefore, we can take the following one,
	\[\Omega(\Lambda)\stackrel{j}\hookrightarrow  \Lambda\otimes_k  \Lambda\stackrel{\mu}{\twoheadrightarrow} \Lambda,\]
	where $\mu\colon  \Lambda\otimes_k  \Lambda\rightarrow  \Lambda$ is the product in $ \Lambda$ and $j$ is the inclusion of $\Omega(\Lambda)=\ker\mu$. With these choices, we have the following commutative diagram
	\begin{center}
		\begin{tikzcd}
		\Omega(\Lambda)\otimes  \Omega(\Lambda)\ar[r, hook, "j\otimes 1"] \ar[d, "-1"'] & \Lambda\otimes_k\Lambda \otimes  \Omega(\Lambda) \ar[r, two heads, "\mu\otimes 1"] \ar[d, "f"] & \Lambda\otimes  \Omega(\Lambda)\cong \Omega(\Lambda)\ar[d, equal, xshift=25]\\
		\Omega(\Lambda)\otimes  \Omega(\Lambda)\ar[r, hook, "1\otimes j"] & \Omega(\Lambda)\otimes  \Lambda\otimes_k\Lambda  \ar[r, two heads, "1\otimes\mu"] & \Omega(\Lambda)\otimes  \Lambda\cong \Omega(\Lambda)
		\end{tikzcd}
	\end{center}
	where $f$ is defined as
	\[\textstyle f\left(a\otimes b\otimes \left(\sum_ic_i\otimes d_i\right)\right)=\sum_i\left(abc_i\otimes 1-a\otimes bc_i\right)\otimes 1\otimes d_i.\]
	This proves the claim.
\end{proof}

We can similarly choose an adjoint inverse cosyzygy functor by taking a short exact sequence in $\bimp{\Lambda}$
\[\Lambda\stackrel{j'}\hookrightarrow P'\stackrel{p'}\twoheadrightarrow \Omega^{-1}(\Lambda)\]
with $P'$ projective-injective and setting $\Omega^{-1}(M)=\Omega^{-1}(\Lambda)\otimes M$. 
The counit is determined by an isomorphism in $\bimpst{\Lambda}$
\[\xi\colon \Omega(\Lambda) \otimes \Omega^{-1}(\Lambda)\cong\Lambda\]
given by any choice of map of extensions in $\bimp{\Lambda}$ as follows
\begin{center}
	\begin{tikzcd}
	\Omega(\Lambda) \otimes \Omega^{-1}(\Lambda) \ar[r, hook, "j\otimes 1"] \ar[d] & P \otimes \Omega^{-1}(\Lambda) \ar[r, two heads, "p\otimes 1"] \ar[d] & \Lambda \otimes \Omega^{-1}(\Lambda) \cong &[-31pt]\Omega^{-1}(\Lambda) \ar[d, equal] \\
	\Lambda \ar[r, hook, "j'"] & P' \ar[rr, two heads, "p'"] && \Omega^{-1}(\Lambda)
	\end{tikzcd}
\end{center}
Once $\xi$ is chosen, there is only one possible isomorphism in $\bimpst{\Lambda}$
\[\xi'\colon \Lambda\cong\Omega^{-1}(\Lambda) \otimes \Omega(\Lambda)\]
defining the unit. It is easy to check, using the coherence of $\zeta$ and Lemma \ref{sign}, that $\xi'=-\zeta_{\Omega^{-1}(\Lambda)}\xi^{-1}$. 
%
%
%
%
There are also coherent isomorphisms in $\bimpst{\Lambda}$
\begin{equation}\label{zeta2}
\zeta_M'\colon M\otimes \Omega^{-1}(\Lambda)\cong \Omega^{-1}(\Lambda)\otimes  M,
\end{equation}
like those in \eqref{zeta}. The isomorphisms $\zeta_{M}$ and $\zeta_{M}'$ are compatible in the sense that the two isomorphisms $\Omega(\Lambda)\otimes  M\otimes \Omega^{-1}(\Lambda)\cong M$ which can be constructed by using these isomorphisms and $\xi$ coincide.

Now, if we define $\Omega^n(\Lambda)$, $n\in\mathbb Z$, as a tensor power of $\Omega(\Lambda)$ or $\Omega^{-1}(\Lambda)$, according to the sign of $n$, the isomorphisms $\xi$ and $\xi'$ define associative isomorphisms $\Omega^p(\Lambda)\otimes\Omega^q(\Lambda)\cong \Omega^{p+q}(\Lambda)$, $p,q\in\mathbb Z$.

Given two objects $M$ and $N$ in $\bimpst{\Lambda}$, there is a natural associative \emph{cup-product} operation
\[\htate{p}{\Lambda,M}\otimes_k \htate{q}{\Lambda,N}\longrightarrow\htate{p+q}{\Lambda,M\otimes N}\]
defined by
\[(\Omega^p(\Lambda) \stackrel{f}\rightarrow M)\smile(\Omega^q(\Lambda)\stackrel{g}\rightarrow N)=(\Omega^{p+q}(\Lambda)
\cong \Omega^{p}(\Lambda)\otimes\Omega^{q}(\Lambda)\stackrel{f\otimes g}\longrightarrow M\otimes N).\]
In particular, $\htate{\star}{\Lambda,M}$ is a graded algebra if $M$ is a monoid in $\bimpst{\Lambda}$, even a bigraded algebra if $M$ is graded. The comparison maps from Hochschild to stable Hochschild cohomology preserve the cup-product.

Let $\sigma\colon\Lambda\rightarrow\Lambda$ be an algebra automorphism. We will need some knowledge on units in the bigraded algebra $\hh{\star,*}{\Lambda,\Lambda(\sigma)}$.

\begin{proposition}\label{surjective_units}
	The comparison map $\hh{0}{\Lambda,\Lambda}\twoheadrightarrow \htate{0}{\Lambda,\Lambda}$ (co)restricts to a surjection between groups of units $\hh{0}{\Lambda,\Lambda}^\times\twoheadrightarrow \htate{0}{\Lambda,\Lambda}^\times$.
\end{proposition}

This follows from \cite[Corollary 2.3]{chen_surjections_2019} since $\hh{0}{\Lambda,\Lambda}$ is the center of $\Lambda$, and hence a finite-dimensional commutative algebra.

\begin{proposition}\label{units}
	An element $f\in \htate{p,q}{\Lambda,\Lambda(\sigma)}=\htate{p}{\Lambda,{}_{\sigma^q}\Lambda_1}$ is a unit in $\htate{\star,*}{\Lambda,\Lambda(\sigma)}$ if and only if $f\colon \Omega^p(\Lambda)\rightarrow {}_{\sigma^q}\Lambda_1$ is an isomorphism in $\bimpst{\Lambda}$.
\end{proposition}

\begin{proof}
	Given $g\colon \Omega^{p'}(\Lambda)\rightarrow {}_{\sigma^{q'}}\Lambda_1$, if we use the isomorphisms $\Omega^p(\Lambda)\otimes\Omega^{p'}(\Lambda)\cong \Omega^{p+p'}(\Lambda)$ as identifications, the cup-product $f\smile g$ above coincides with $(f\otimes{}_{\sigma^{q'}}\Lambda_1)(\Omega^{p}(\Lambda)\otimes g)=({}_{\sigma^q}\Lambda_1\otimes g)(f\otimes\Omega^{p'}(\Lambda))$. 
	
	Assume $f$ is invertible. Then $1_\Lambda=f\smile f^{-1}=(f\otimes{}_{\sigma^{-q}}\Lambda_1)(\Omega^{p}(\Lambda)\otimes f^{-1})=({}_{\sigma^q}\Lambda_1\otimes f^{-1})(f\otimes\Omega^{-p}(\Lambda))$. The functor $\Omega^p=\Omega^p(\Lambda)\otimes-$ is an equivalence, and so is $-\otimes\Omega^{-p}(\Lambda)$, which is naturally isomorphic to $\Omega^{-p}$ via $\zeta$ or $\zeta'$. Moreover, ${}_{\sigma^q}\Lambda_1\otimes-$ is also an equivalence of categories since ${}_{\sigma^q}\Lambda_1$ is an invertible $\Lambda$-bimodule with inverse ${}_{\sigma^{-q}}\Lambda_1$, and similarly $-\otimes {}_{\sigma^{-q}}\Lambda_1$. Therefore, $f$ is both a split epimorphism and a split monomorphism in $\bimpst{\Lambda}$, and hence an isomorphism.
	
	Conversely, if $f$ is an isomorphism in $\bimpst{\Lambda}$, then so is $f\otimes{}_{\sigma^{-q}}\Lambda_1$, and the inverse isomorphism $\Omega^{-p}(f\otimes{}_{\sigma^{-q}}\Lambda_1)^{-1}$ is a cup-product inverse for $f$.
\end{proof}

\begin{remark}\label{units_as_extensions}
	For any $p>0$ and $f\in \htate{p,q}{\Lambda,\Lambda(\sigma)}=\ext_{\Lambda}^p(\Lambda,{}_{\sigma^q}\Lambda_1)$, we can take a representing extension
	\[{}_{\sigma^q}\Lambda_1\hookrightarrow P_{p-1}\rightarrow\cdots\rightarrow P_0\twoheadrightarrow\Lambda\]
	with $P_{p-2},\dots, P_0$ projective-injective. Using Proposition \ref{units}, we see that $f$ is a unit in $\htate{\star,*}{\Lambda,\Lambda(\sigma)}$ if and only if $P_{p-1}$ is also projective-injective.
\end{remark}

	We can also use the previous choice of $\Omega(\Lambda)$ to fix a syzygy functor in the stable category of right modules $\modulesst{\Lambda}$, namely $\Omega(M)=M\otimes\Omega(\Lambda)$, and similarly for the cosyzygy functor $\Omega^{-1}$. The extended Yoneda product in $\modulesst{\Lambda}$,
	\[\exttate_{\Lambda}^p(M,N)\otimes_k \exttate_{\Lambda}^q(L,M)\longrightarrow
	\exttate_{\Lambda}^{p+q}(L,N),\]
	can be computed as follows,
	\begin{multline*}
	\homst_{\Lambda}(\Omega^p(M),N)\otimes_k \homst_{\Lambda}(\Omega^q(L),M)\longrightarrow
	\homst_{\Lambda}(\Omega^{p+q}(L),N),\\
	(M\otimes \Omega^p(\Lambda) \stackrel{f}\rightarrow N)\otimes (L\otimes \Omega^q(\Lambda)\stackrel{g}\rightarrow M)\mapsto\qquad\qquad\qquad\qquad\\
	(L\otimes \Omega^{p+q}(\Lambda)\cong L\otimes \Omega^q(\Lambda)\otimes \Omega^{p}(\Lambda)\stackrel{g\otimes 1}{\longrightarrow}M\otimes \Omega^p(\Lambda)\stackrel{f}{\rightarrow}N).
	\end{multline*}
	
\begin{remark}\label{yoneda_extended}
	As above, $f\in \exttate_{\Lambda}^p(M,N)$ is a Yoneda unit if and only if $f\colon \Omega^p(M)\rightarrow N$ is an isomorphism in $\modulesst{\Lambda}$. For $p>0$, $\exttate_{\Lambda}^p(M,N)=\ext_{\Lambda}^p(M,N)$ so we can take an extension
	\[N\rightarrow P_{p-1}\rightarrow\cdots\rightarrow P_0\twoheadrightarrow M\]
	representing $f$ with $P_{p-2},\dots, P_0$ projective-injective, and $f$ is a Yoneda unit if and only if $P_{p-1}$ is also projective-injective.
\end{remark}

For any $\Lambda$-bimodule $M$, there is a natural map
\begin{equation*}\label{mapa}
	\htate{n}{\Lambda,M}\longrightarrow\hh{0}{\modulesfp{\Lambda},\exttate_{\Lambda}^n(-,-\otimes M)}
\end{equation*}
defined as
\[(\Omega^n(\Lambda)\stackrel{f}{\rightarrow}M)\mapsto 
\left(L\mapsto \left(\Omega^n(L)=L\otimes\Omega^n(\Lambda)\stackrel{L\otimes f}{\longrightarrow}L\otimes M\right)\right).\]
Let $\sigma\colon\Lambda\rightarrow\Lambda$ be an automorphism. 
These maps 
assemble to a map
\begin{equation}\label{epsilon}
\varepsilon\colon \htate{\star,*}{\Lambda,\Lambda(\sigma)}\longrightarrow\hh{0,*}{\modulesfp{\Lambda},\exttate_{\Lambda}^{\star,*}(-,-\otimes\Lambda(\sigma))}.
\end{equation}
The coefficient $\modulesfp{\Lambda}$-bimodule on the right is a monoid for the Yoneda composition. Indeed, since each ${}_{\sigma^q}\Lambda_1$ is left and right projective, the functor $-\otimes {}_{\sigma^q}\Lambda_1$ is exact (actually, an exact equivalence since ${}_{\sigma^q}\Lambda_1$ is invertible) and the product of $g\in \exttate_{\Lambda}^{p,q}(L,M\otimes\Lambda(\sigma))=\exttate_{\Lambda}^{p}(L,M\otimes {}_{\sigma^q}\Lambda_1)$ and $f\in \exttate_{\Lambda}^{s,t}(M,N\otimes\Lambda(\sigma))=
\exttate_{\Lambda}^{s}(M,N\otimes {}_{\sigma^t}\Lambda_1)$ is the Yoneda product of $g$ and $f\otimes{}_{\sigma^q}\Lambda_1$. In particular, an element in the target is a unit if and only if it is pointwise a Yoneda unit.

\begin{remark}\label{coefficient_bimodule}
We can regard $\modulesfp{\Lambda}$ as the degree $0$ part of $\modulesfp{\Lambda(\sigma)}$, compare \cite[\S3]{muro_first_2019}. 
Therefore, $\modulesfp{\Lambda(\sigma)}$ is a graded Frobenius abelian category and graded Tate $\exttate_{\Lambda(\sigma)}^{n,*}$ functors are defined for all $n\in\mathbb Z$ (also the extended Yoneda product). The inclusion $\modulesfp{\Lambda}\subset \modulesfp{\Lambda(\sigma)}$ corresponds to the extension of scalars along $\Lambda\subset\Lambda(\sigma)$. Therefore, $\exttate_{\Lambda}^{n,*}(-,-\otimes\Lambda(\sigma))$ coincides with $\exttate_{\Lambda(\sigma)}^{n,*}$ as a monoid in $\modulesfp{\Lambda}$-bimodules (also in the non-Tate case), and we can exchange them in the target of \eqref{epsilon}. 
\end{remark}

\begin{proposition}\label{edge_units}
	If $k$ is perfect, then the morphism $\varepsilon$ in \eqref{epsilon} preserves and reflects units for $\star>0$.
\end{proposition}

\begin{proof}
	Recall from Remark \ref{units_as_extensions} that any $x\in \htate{n,s}{\Lambda,\Lambda(\sigma)}$ with $n>0$ can be represented by an extension of finite-dimensional $\Lambda$-bimodules
	\[{}_{\sigma^s}\Lambda_1\hookrightarrow P_{n-1}\rightarrow\cdots\rightarrow P_0\twoheadrightarrow\Lambda\]
	with $P_0,\dots, P_{n-2}$ projective. Moreover, $x$ is a unit if and only if $P_{n-1}$ is also projective. All bimodules in this sequence are projective as left or right $\Lambda$-modules. 
	
	For each finitely presented right $\Lambda$-module $M$, $\varepsilon(x)(M)\in \exttate^{n}_{\Lambda}(M,M\otimes {}_{\sigma^s}\Lambda_1)$ is represented by the extension
	\[M\otimes{}_{\sigma^s}\Lambda_1\hookrightarrow M\otimes P_{n-1}\rightarrow\cdots\rightarrow M\otimes P_0\twoheadrightarrow  M\otimes\Lambda\cong M\]
	obtained by applying $M\otimes-$ to the bimodule extension. All right $\Lambda$-modules $M\otimes P_i$ are projective for $0\leq i\leq n-2$. Moreover, $\varepsilon(x)$ is a unit if and only if each $\varepsilon(x)(M)$ is a Yoneda unit, and this happens if and only if $M\otimes P_{n-1}$ is also a projective right $\Lambda$-module for all $M$, see Remark \ref{yoneda_extended}. This clearly holds if $P_{n-1}$ is projective. The converse is also true when $k$ is perfect by \cite[Theorem 3.1]{auslander_theorem_1991}.	
\end{proof}

\section{Edge units}\label{edge_units_section}

Throughout this section, we will assume that the ground field $k$ is a perfect field, since we will derive consequences of Proposition \ref{edge_units}. Moreover, $\Lambda$ is a Frobenius algebra and $\sigma\colon\Lambda\to\Lambda$ is an algebra automorphism. Recall also that, in this section, like in the previous one, $\otimes$ stands for $\otimes_{\Lambda}$.

We defined in \cite[Proposition 5.1]{muro_first_2019} a first quadrant spectral sequence
\begin{equation}\label{spectral_graded}
	E_2^{p,q}=\hh{p,*}{\modulesfp{\Lambda(\sigma)},\ext_{\Lambda(\sigma)}^{q,*}}\Longrightarrow\hh{p+q,*}{\Lambda(\sigma),\Lambda(\sigma)}.
\end{equation}
It has edge morphisms
\[\hh{\star,*}{\Lambda(\sigma),\Lambda(\sigma)}\longrightarrow
\hh{0,*}{\modulesfp{\Lambda(\sigma)},{\ext}^{\star,*}_{\Lambda(\sigma)}}.\]

\begin{definition}
	An element in $\hh{\star,*}{\Lambda(\sigma),\Lambda(\sigma)}$ is an \emph{edge unit} if its image along 
	\[\hh{\star,*}{\Lambda(\sigma),\Lambda(\sigma)}\stackrel{\text{edge}}\longrightarrow
	\hh{0,*}{\modulesfp{\Lambda(\sigma)},{\ext}^{\star,*}_{\Lambda(\sigma)}}
	\!\!\stackrel{\text{comp.}}{\longrightarrow}\!\!\hh{0,*}{\modulesfp{\Lambda(\sigma)},\exttate^{\star,*}_{\Lambda(\sigma)}}\]
	is a unit.
\end{definition}

Note that it is really crucial to get to Tate $\ext$ coefficients. Otherwise, there would only be  edge units in horizontal degree $\star=0$.

There is another spectral sequence
\begin{equation}\label{spectral_ungraded}
E_2^{p,q}=\hh{p,*}{\modulesfp{\Lambda},\ext_{\Lambda(\sigma)}^{q,*}}
\Longrightarrow
\hh{p+q,*}{\Lambda,\Lambda(\sigma)}
\end{equation}
with ungraded first variables, see \cite[Remark 5.4]{muro_first_2019}.  It has edge morphisms
\begin{equation*}
\hh{\star,*}{\Lambda,\Lambda(\sigma)}\longrightarrow
\hh{0,*}{\modulesfp{\Lambda},{\ext}^{\star,*}_{\Lambda(\sigma)}}.
\end{equation*}

\begin{definition}
	An element in $\hh{\star,*}{\Lambda,\Lambda(\sigma)}$ is an \emph{edge unit} if its image along \[\hh{\star,*}{\Lambda,\Lambda(\sigma)}\stackrel{\text{edge}}\longrightarrow
	\hh{0,*}{\modulesfp{\Lambda},{\ext}^{\star,*}_{\Lambda(\sigma)}}
	\stackrel{\text{comp.}}{\longrightarrow}\hh{0,*}{\modulesfp{\Lambda},\exttate^{\star,*}_{\Lambda(\sigma)}}\]
	is a unit.
\end{definition}

These edge units have the following smarter characterization in positive Hochschild degree. 

\begin{proposition}\label{edge_characterization}
	For $n>0$, an element in $\hh{n,*}{\Lambda,\Lambda(\sigma)}$ is an edge unit if and only if its image along the comparison morphism \[\hh{n,*}{\Lambda,\Lambda(\sigma)}\longrightarrow\htate{n,*}{\Lambda,\Lambda(\sigma)}\]
	is a unit in $\htate{\star,*}{\Lambda,\Lambda(\sigma)}$.
\end{proposition}

\begin{proof}
	By \cite[Remark 5.7]{muro_first_2019}, the following square commutes, 
	\begin{equation*}
	\begin{tikzcd}
	\hh{n,*}{\Lambda,\Lambda(\sigma)}\arrow[r,"\text{\scriptsize edge}"]\arrow[d,"\text{comparison}"']&\hh{0,*}{\modulesfp{\Lambda},\ext^{n,*}_{\Lambda(\sigma)}}\arrow[d,"\text{comparison}"]\\
	\htate{n,*}{\Lambda,\Lambda(\sigma)}\arrow[r,"\varepsilon"]&\hh{0,*}{\modulesfp{\Lambda},\exttate^{n,*}_{\Lambda(\sigma)}}
	\end{tikzcd}
	\end{equation*}
	Here, the bottom horizontal arrow is the horizontal degree $n$ part of \eqref{epsilon} (see Remark \ref{coefficient_bimodule}),  which preserves and reflects units for $n>0$ by Proposition \ref{edge_units}. Hence, we are done.
\end{proof}

\begin{proposition}\label{ss_comparison_edge}
	The morphism induced by the inclusion $i\colon\Lambda\subset\Lambda(\sigma)$,
	\[i^*\colon \hh{n,*}{\Lambda(\sigma), \Lambda(\sigma)}\longrightarrow\hh{n,*}{\Lambda, \Lambda(\sigma)},\]
	preserves and reflects edge units.
\end{proposition}

\begin{proof}
	By \cite[Remark 7.4]{muro_first_2019}, there is a morphism of spectral sequences from \eqref{spectral_graded} to \eqref{spectral_ungraded} induced by the inclusion of the degree $0$ part $\Lambda\subset\Lambda(\sigma)$, hence we have a commutative diagram
	\begin{center}
		\begin{tikzcd}[column sep = 11pt]
		\hh{\star,*}{\Lambda(\sigma),\Lambda(\sigma)}\ar[r,"\text{edge}"]\ar[d]&
		\hh{0,*}{\modulesfp{\Lambda(\sigma)},{\ext}^{\star,*}_{\Lambda(\sigma)}}
		\ar[r,"\text{comp.}"]\ar[d, hook]&\hh{0,*}{\modulesfp{\Lambda(\sigma)},\exttate^{\star,*}_{\Lambda(\sigma)}}\ar[d, hook]\\
		\hh{\star,*}{\Lambda,\Lambda(\sigma)}\ar[r,"\text{edge}"]&
		\hh{0,*}{\modulesfp{\Lambda},{\ext}^{\star,*}_{\Lambda(\sigma)}}
		\ar[r,"\text{comp.}"]&\hh{0,*}{\modulesfp{\Lambda},\exttate^{\star,*}_{\Lambda(\sigma)}}
		\end{tikzcd}
	\end{center}
	whose vertical arrows are induced by $\Lambda\subset\Lambda(\sigma)$. The two last vertical arrows are injective because they are in Hochschild degree $0$, see Proposition \ref{graded_ungraded_long_exact_sequence}. Moreover, the last one preserves and reflects units by the first paragraph of the proof of \cite[Corollary 3.4]{muro_first_2019}. This suffices.
\end{proof}

\begin{corollary}
	For $n>0$, an element in $\hh{n,*}{\Lambda(\sigma),\Lambda(\sigma)}$ is an edge unit if and only if its image along 
	\[\hh{n,*}{\Lambda(\sigma),\Lambda(\sigma)}\stackrel{i^*}\longrightarrow
	\hh{n,*}{\Lambda,\Lambda(\sigma)}
	\stackrel{\text{comp.}}{\longrightarrow}\htate{n,*}{\Lambda,\Lambda(\sigma)}\]
	is a unit in $\htate{\star,*}{\Lambda,\Lambda(\sigma)}$.
\end{corollary}

In order to conclude this section, we record a property of edge units which is crucial for later computations.

\begin{lemma}\label{non_singularity}
	If $x\in\hh{3,-1}{\Lambda(\sigma),\Lambda(\sigma)}$ is an edge unit, then the map
	\[x\cdot -\colon \hh{p,q}{\Lambda(\sigma),\Lambda(\sigma)}\longrightarrow\hh{p+3,q-1}{\Lambda(\sigma),\Lambda(\sigma)}\]
	given by left multiplication by $x$ 
	is an isomorphism for all $p\geq 2$ and $q\in\mathbb Z$ and an epimorphism for $p=1$ and all $q\in\mathbb Z$.
\end{lemma}

\begin{proof}
	By the previous corollary $i^*(x)\in \hh{3,-1}{\Lambda,\Lambda(\sigma)}=\htate{3,-1}{\Lambda,\Lambda(\sigma)}$ is a unit in $\htate{\star,*}{\Lambda,\Lambda(\sigma)}$. 
	Since the Hochschild--Tate cohomology coincides with Hochschild's in  positive Hochschild degrees, and the latter surjects onto the former in Hochschild degree $0$, the map 
	\[i^*(x) \cdot-\colon \hh{p,q}{\Lambda,\Lambda(\sigma)}\longrightarrow\hh{p+3,q-1}{\Lambda,\Lambda(\sigma)}\]
	is an isomorphism for all $p\geq 1$ and $q\in\mathbb Z$ and an epimorphism for $p=0$ and all $q\in\mathbb Z$.
	
	By Proposition \ref{graded_ungraded_long_exact_sequence}, we have a map of long exact sequences
	\begin{center}
		\begin{tikzcd}[row sep=5mm, column sep=40pt]
		\vdots\arrow[d]&\vdots\arrow[d]\\
		\hh{p,q}{\Lambda(\sigma),\Lambda(\sigma)}\arrow[d, "i^*"]\arrow[r,"x \cdot-"]&\hh{p+3,q-1}{\Lambda(\sigma),\Lambda(\sigma)}\arrow[d, "i^*"]\\
		\hh{p,q}{\Lambda,\Lambda(\sigma)}\arrow[d, "\operatorname{id}-\sigma_*^{-1}\sigma^*"]\arrow[r,"i^*(x) \cdot-"]&\hh{p+3,q-1}{\Lambda,\Lambda(\sigma)}\arrow[d, "\operatorname{id}-\sigma_*^{-1}\sigma^*"]\\ 
		\hh{p,q}{\Lambda,\Lambda(\sigma)}\arrow[d,"\partial"]\arrow[r,"i^*(x) \cdot-"]&\hh{p+3,q-1}{\Lambda,\Lambda(\sigma)}\arrow[d,"\partial"]\\
		\hh{p+1,q}{\Lambda(\sigma),\Lambda(\sigma)}\arrow[d]\arrow[r,"x \cdot-"]&\hh{p+4,q-1}{\Lambda(\sigma),\Lambda(\sigma)}\arrow[d]\\
		\vdots&\vdots
		\end{tikzcd}
	\end{center}
	Hence, the proposition follows from the previous paragraph and the five lemma.
\end{proof}

\section{Lifting edge units}\label{restricted}

The goal of this section is to lift edge units along the morphism in Proposition \ref{ss_comparison_edge}, mainly in bidegree $(3,-1)$. We will use the characterization of edge units in Proposition \ref{edge_characterization}, as well as other consequences of Proposition \ref{edge_units}. Therefore, in this section the ground field $k$ will be a perfect field. We start with some consequences of results from Section \ref{hochschild_tate}.

Let $\Lambda$ be a Frobenius algebra and $\sigma\colon\Lambda\rightarrow\Lambda$ an automorphism. Recall that $i\colon\Lambda\subset\Lambda(\sigma)$ denotes the inclusion of the degree $0$ part. 
In the rest of this section, we will work in the monoidal triangulated category $\bimpst{\Lambda}$ introduced in Section \ref{hochschild_tate}, and $\otimes$ will exceptionally stand for $\otimes_{\Lambda}$. 

Recall the natural isomorphism $\zeta_{M}$ from \eqref{zeta}. The following result is an immediate consequence of Lemma \ref{sign}.

\begin{corollary}\label{sign2}
	For $n\geq 1$, the automorphism $\zeta_{\Omega^n(\Lambda)}$ of $\Omega(\Lambda)^{\otimes ^{n+1}}$ is $(-1)^n$.
\end{corollary}

\begin{corollary}
	If $f\colon \Omega^n(\Lambda)\rightarrow M$ is a map, $n\geq 1$, then the following diagram commutes,
	\begin{center}
		\begin{tikzcd}
		\Omega(\Lambda)\otimes \Omega(\Lambda)^{\otimes ^n}\ar[r, "(-1)^n"]\ar[d, "1\otimes f"]&
		\Omega(\Lambda)^{\otimes ^n}\otimes \Omega(\Lambda)
		\ar[d, "f\otimes 1"]
		\\
		\Omega(\Lambda)\otimes M\ar[r, "\zeta_{M}"]&
		M\otimes \Omega(\Lambda)
		\end{tikzcd}
	\end{center}
\end{corollary}

This follows by naturality.

\begin{corollary}
	If $f\colon \Omega^n(\Lambda)\rightarrow M$ is a map, $n\geq 1$, then the following diagram commutes in $\bimpst{\Lambda}$,
	\begin{center}
		\begin{tikzcd}[column sep = 80]
		\Omega(\Lambda)^{\otimes ^n}\otimes \Omega(\Lambda)^{\otimes^n}\ar[r, "(-1)^{n^2}"]\ar[d, "1\otimes f"]&
		\Omega(\Lambda)^{\otimes ^n}\otimes \Omega(\Lambda)^{\otimes ^n}
		\ar[d, "f\otimes 1"]
		\\
		\Omega(\Lambda)^{\otimes ^n}\otimes M\ar[r, "\prod_{i=1}^n1^{\otimes^{i-1}}\otimes\zeta_{M}\otimes 1^{\otimes^{n-i}}"]&
		M\otimes \Omega(\Lambda)^{\otimes ^n}
		\end{tikzcd}
	\end{center}
\end{corollary}

We use again naturality here.

\begin{corollary}
	If $f\colon \Omega^n(\Lambda)\rightarrow M$ is an isomorphism in $\bimpst{\Lambda}$, $n\geq 1$, then the following diagram commutes in this category,
	\begin{center}
		\begin{tikzcd}[column sep = 90]
		\Omega(\Lambda)^{\otimes ^n}\otimes M\ar[r, "\prod_{i=1}^n1^{\otimes^{i-1}}\otimes\zeta_{M}\otimes 1^{\otimes^{n-i}}"] \ar[d, "f\otimes 1"] &
		M\otimes \Omega(\Lambda)^{\otimes ^n} \ar[d, "1\otimes f"] \\
		M\otimes M\ar[r, "(-1)^{n^2}"] & M\otimes M
		\end{tikzcd}
	\end{center}
\end{corollary}

\begin{proof}
	It suffices to consider the following commutative diagram of isomorphisms,
	\begin{center}
		\begin{tikzcd}[column sep = 80 ]
		 \Omega(\Lambda)^{\otimes ^n}\otimes \Omega(\Lambda)^{\otimes^n}\ar[r, "(-1)^{n^2}"]\ar[d, "1\otimes f"'] \ar[dd, bend right = 75, "f\otimes f"'] &
		\Omega(\Lambda)^{\otimes ^n}\otimes \Omega(\Lambda)^{\otimes ^n}
		\ar[d, "f\otimes 1"] \ar[dd, bend left = 75, "f\otimes f"] 
		\\ 
		\Omega(\Lambda)^{\otimes ^n}\otimes M\ar[r, "\prod_{i=1}^n1^{\otimes^{i-1}}\otimes\zeta_{M}\otimes 1^{\otimes^{n-i}}"] \ar[d, "f\otimes 1"'] &
		M\otimes \Omega(\Lambda)^{\otimes ^n} \ar[d, "1\otimes f"]  \\
		M\otimes M&M\otimes M
		\end{tikzcd}
	\end{center}
\end{proof}

The functor
\[{}_\sigma\Lambda_1\otimes -\otimes {}_1\Lambda_\sigma\colon \bimp{\Lambda}\longrightarrow \bimp{\Lambda}\]
is an exact equivalence, so it gives rise to a triangulated equivalence
\[{}_\sigma\Lambda_1\otimes -\otimes {}_1\Lambda_\sigma\colon \bimpst{\Lambda}\longrightarrow \bimpst{\Lambda}.\] 
Part of this triangulated equivalence is a natural isomorphism in $\bimpst{\Lambda}$, which shows that this functor commutes with the syzygy functor, 
\[\omega_{M}\colon\Omega({}_\sigma\Lambda_1\otimes M\otimes {}_1\Lambda_\sigma)\cong {}_\sigma\Lambda_1\otimes \Omega(M)\otimes {}_1\Lambda_\sigma.\]
It is simply given by
\[ \omega_M=\zeta_{{}_\sigma\Lambda_1}\otimes 1_{M\otimes {}_1\Lambda_\sigma}\colon 
\Omega(\Lambda)\otimes {}_\sigma\Lambda_1\otimes M\otimes {}_1\Lambda_\sigma
\cong 
{}_\sigma\Lambda_1\otimes \Omega(\Lambda)\otimes M\otimes {}_1\Lambda_\sigma.\]

We also have an exact equivalence
\[\bimp{\Lambda}\longrightarrow \bimp{\Lambda}\colon M\mapsto {}_\sigma M_\sigma,\]
which is exactly the same thing as the restriction of scalars along $\sigma\colon\Lambda\rightarrow\Lambda$. It is naturally isomorphic to the previous one via the map
\[\rho_M\colon {}_\sigma\Lambda_1\otimes M\otimes {}_1\Lambda_\sigma\cong {}_\sigma M_\sigma,
\qquad \rho_M(a\otimes x\otimes b)=axb.
\]
In addition, we obtain an induced triangulated equivalence
\[\bimpst{\Lambda}\longrightarrow \bimpst{\Lambda}\colon M\mapsto {}_\sigma M_\sigma,\]
also naturally isomorphic to the previous one in the same way, as triangulated functors. Again, part of this triangulated equivalence is a natural isomorphism in $\bimpst{\Lambda}$, which shows that this functor commutes with the syzygy functor, 
\[\bar\omega_{M}\colon \Omega({}_\sigma M_\sigma)\cong {}_\sigma\Omega(M)_\sigma.\]
The natural isomorphism of triangulated equivalences is compatible with the natural isomorphisms $\omega$ and $\bar\omega$, in the sense that the following square of natural isomorphisms commutes,
\begin{equation*}
	\begin{tikzcd}
		\Omega({}_\sigma\Lambda_1\otimes M\otimes {}_1\Lambda_\sigma)\ar[r, "\omega_M"]\ar[d,"\Omega(\rho_M)"']& {}_\sigma\Lambda_1\otimes \Omega(M)\otimes {}_1\Lambda_\sigma\ar[d,"\rho_{\Omega(M)}"]\\
		\Omega({}_\sigma M_\sigma)\ar[r, "\bar \omega_M"]&{}_\sigma\Omega(M)_\sigma
	\end{tikzcd}
\end{equation*}

Let us denote 
\begin{align*}
	\omega_{M}^{(n)}&=\prod_{i=1}^n\Omega^{i-1}(\omega_{\Omega^{n-i}(M)}),& \bar\omega_{M}^{(n)}&=\prod_{i=1}^n\Omega^{i-1}(\bar \omega_{\Omega^{n-i}(M)}).
\end{align*}
These are isomorphisms fitting in a similar diagram,
\begin{equation}\label{diag}
\begin{tikzcd}
\Omega^n({}_\sigma\Lambda_1\otimes M\otimes {}_1\Lambda_\sigma)\ar[r, "\omega_M^{(n)}"]\ar[d,"\Omega^n(\rho_M)"']& {}_\sigma\Lambda_1\otimes \Omega^n(M)\otimes {}_1\Lambda_\sigma\ar[d,"\rho_{\Omega^n(M)}"]\\
\Omega^n({}_\sigma M_\sigma)\ar[r, "\bar \omega_M^{(n)}"]&{}_\sigma\Omega^n(M)_\sigma
\end{tikzcd}
\end{equation}

\begin{proposition}\label{complicada}
	If $f\colon \Omega(\Lambda)^{\otimes^n}=\Omega^n(\Lambda)\rightarrow {}_\sigma\Lambda_1$ is an isomorphism in $\bimpst{\Lambda}$, $n\geq 1$, then the following diagram commutes in this category,
	\begin{center}
		\begin{tikzcd}
		\Omega^n(\Lambda) \ar[r, "\Omega^n(\sigma)"] & \Omega^n({}_\sigma\Lambda_\sigma) \ar[r, "\bar\omega_{\Lambda}^n"] & {}_\sigma\Omega^n(\Lambda)_\sigma \ar[d, "{{}_\sigma f_\sigma}"]\\
		{}_\sigma\Lambda_1 \ar[u, <-, "f"] \ar[rr, "(-1)^{n^2}\sigma"] & & 
		{}_{\sigma^2}\Lambda_{\sigma}
		\end{tikzcd}
	\end{center}
\end{proposition}

\begin{proof}
	By the previous corollaries and various naturality properties, we have the following commutative diagram of isomorphisms,
	\begin{center}
		\includegraphics[scale=.7]{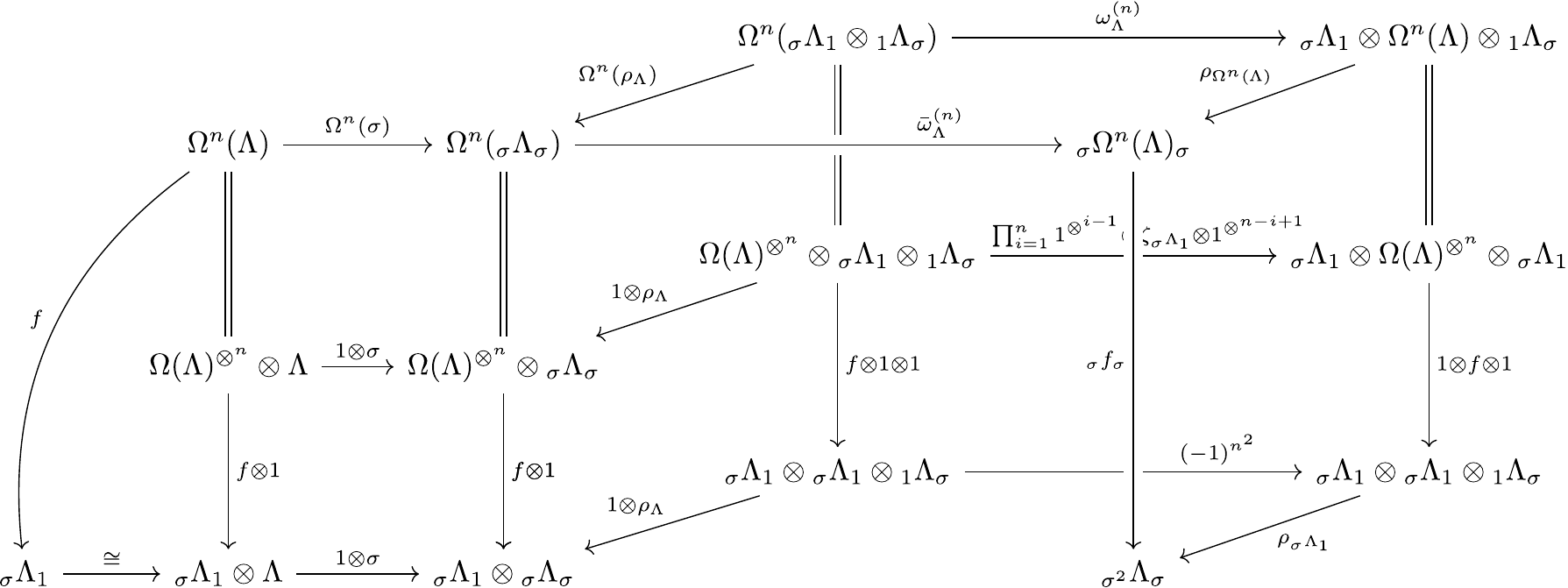}
	\end{center}
It therefore suffices to check that the bottom composite is $(-1)^{n^2}\sigma$. This is very easy to check. 
Actually, if we recall that $\rho$ is defined by multiplication, we readily see that, starting with $a\in{}_\sigma\Lambda_1$ in the bottom left corner,
\begin{center}
	\begin{tikzcd}[column sep = 10, row sep = 10]
	&&&1\otimes\sigma(a)\otimes1
	\ar[rr, mapsto] && (-1)^{n^2}1\otimes\sigma(a)\otimes1 \ar[ld, mapsto] \\
	a
	\ar[r, mapsto] &
	1\otimes a
	\ar[r, mapsto] &
	1\otimes\sigma(a)
	\ar[ru, mapsto] && 	(-1)^{n^2}\sigma(a)
	\end{tikzcd}
\end{center}
\end{proof}

\begin{corollary}
	For $n>0$, if $n$ is odd or $\characteristic k=2$, then \[i^*\colon \hh{n,1}{\Lambda(\sigma),\Lambda(\sigma)}\longrightarrow\hh{n,1}{\Lambda,\Lambda(\sigma)}\]
	induces a surjection on edge units.
\end{corollary}

\begin{proof}
	This morphism induces a map on edge units by Proposition \ref{ss_comparison_edge}. Consider the long exact sequence in Proposition \ref{graded_ungraded_long_exact_sequence} for $\C{T}_\Sigma=\Lambda(\sigma)$ and $*=1$. Proposition \ref{complicada} shows that any edge unit in the target is in the kernel of $\id{}-\sigma^{-1}_*\sigma^*$. Here we use that $n^2$ is odd and that the graded $\sigma$ is defined on the degree $1$ part of $\Lambda(\sigma)$, which is ${}_\sigma\Lambda_1$, as the ungraded $-\sigma$, see Remark \ref{signos}. Therefore, any edge unit in the target has a preimage, which must also be an edge unit by Proposition \ref{ss_comparison_edge}.
\end{proof}

Since $\Lambda(\sigma)$ is the same as $\Lambda(\sigma^{-1})$ reversing the degrees, we also deduce the following result.

\begin{corollary}\label{preimage}
	For $n>0$, if $n$ is odd or $\characteristic k=2$, then \[i^*\colon \hh{n,-1}{\Lambda(\sigma),\Lambda(\sigma)}\longrightarrow\hh{n,-1}{\Lambda,\Lambda(\sigma)}\]
	induces a surjection on edge units.
\end{corollary}

\begin{proposition}\label{kick}
	Given an edge unit $u\in\hh{3,-1}{\Lambda(\sigma),\Lambda(\sigma)}$, there exists $y\in \hh{2,-1}{\Lambda(\sigma),\Lambda(\sigma)}$ such that $\gsquare(u+\{\delta\}\cdot y)=0$.
\end{proposition}

\begin{proof}
	For any $y$ we have 
	\begin{align*}
	\gsquare(u+\{\delta\}\cdot y)
	&=\gsquare(u)+\gsquare(\{\delta\}\cdot y)+[u,\{\delta\}\cdot y]\\
	&=\gsquare(u)+u\cdot y-\{\delta\}\cdot [u,y]
	\end{align*}
	Here we use the Gerstenhaber algebra laws and Propositions \ref{euler_sub_lie}, \ref{euler_gsquare_2} and \ref{nulo}. Since $u\cdot-\colon \hh{2,-1}{\Lambda(\sigma),\Lambda(\sigma)}\rightarrow\hh{5,-2}{\Lambda(\sigma),\Lambda(\sigma)}$ is an isomorphism by Lemma \ref{non_singularity}, there exists a unique $y$ such that $\gsquare(u)+u\cdot y=0$. Let us fix this $y$. If we prove that $[u,y]=0$, then we will be done. Again, $u\cdot-\colon \hh{4,-2}{\Lambda(\sigma),\Lambda(\sigma)}\rightarrow\hh{7,-3}{\Lambda(\sigma),\Lambda(\sigma)}$ is an isomorphism, so it is enough to check that $u\cdot [u,y]=0$. We have that $[\gsquare(u),u]=[u,[u,u]]=0$ and $u\cdot y=-\gsquare(u)$, therefore $0=[u,u\cdot y]
	=[u,u]\cdot y+u\cdot[u,y]$. Once again, since $u\cdot-\colon \hh{7,-3}{\Lambda(\sigma),\Lambda(\sigma)}\rightarrow\hh{10,-4}{\Lambda(\sigma),\Lambda(\sigma)}$ is an isomorphism, it suffices to prove that $u\cdot [u,u]\cdot y=0$. This follows from
	\begin{align*}
	u\cdot [u,u]\cdot y&=[u,u]\cdot (u\cdot y)\\
	&=2\cdot\gsquare(u)\cdot (-\gsquare(u))\\
	&=-2\cdot\gsquare(u)^2\\
	&=0.
	\end{align*}
	Here we use that $\abs{\gsquare(u)}=5-2=3$ is odd, so its square is $2$-torsion by graded commutativity.
\end{proof}

\begin{proposition}\label{bijection}
	The map
	\[i^*\colon \hh{3,-1}{\Lambda(\sigma),\Lambda(\sigma)}\longrightarrow\hh{3,-1}{\Lambda,\Lambda(\sigma)}\] induces a bijection between the set of edge units $u$ in the source satisfying $\gsquare(u)=0$ and the set of edge units in the target with no extra condition. 
\end{proposition}

\begin{proof}
	The map $i^*$ restricts to edge units by Proposition \ref{ss_comparison_edge}. 
	Let us first check surjectivity. Any edge unit in $\hh{3,-1}{\Lambda,\Lambda(\sigma)}$ has a preimage $u$ by Corollary \ref{preimage}. The element $u+\{\delta\}\cdot y$ in Proposition \ref{kick} is another preimage since $\Lambda$ is ungraded so  $i^*(\{\delta\})=0$ and therefore $i^*(u+\{\delta\}\cdot y)=i^*(u)$. Moreover, $u+\{\delta\}\cdot y$ is also an edge unit by Proposition \ref{ss_comparison_edge}.
	
	We now check injectivity. Let $u,u'\in \hh{3,-1}{\Lambda(\sigma),\Lambda(\sigma)}$ be edge units with $\gsquare(u)=\gsquare(u')=0$ and $i^*(u)=i^*(u')$. The element $z=u'-u$ is in the kernel of $i^*$, hence $\gsquare(z)$ too by Proposition \ref{kernel}. Using the Gerstenhaber algebra laws and Proposition \ref{nulo} we obtain
	\begin{align*}
	0&=\gsquare(u')&						  uz&=[u,\{\delta\}z]+\{\delta\}[u,z]\\
	&=\gsquare(u+z)&							&=\{\delta\}[u,z]\\
	&=\gsquare(u)+[u,z]+\gsquare(z)&			&=-\{\delta\}Sq(z)\\
	&=[u,z]+\gsquare(z),&						&=0.
	\end{align*}
	On the right, we use twice that the kernel of $i^*$ is a square zero ideal, see Proposition \ref{graded_ungraded_long_exact_sequence}, and $i^*(\{\delta\})=0$. By Lemma \ref{non_singularity}, $uz=0$ implies that $z=0$, so $u=u'$.
\end{proof}

\section{Enhancements}\label{enhancements}

Recall that a DG-functor between DG-categories $\C A\rightarrow\C B$ is a \emph{Morita equivalence} if it induces an equivalence between their derived categories $D(\C A)\rightarrow D(\C B)$, or equivalently between their full subcategories of compact objects.

\begin{definition}
	Let $\C T$ be a finite category and $\Sigma\colon\C T\rightarrow\C T$ an automorphism. The set of \emph{enhanced triangulated categories} with underlying suspended category $(\C T,\Sigma)$, denoted by 
	\[\etc{\C T,\Sigma},\]
	is the set of Morita equivalence classes of DG-categories $\C A$ such that the derived category of compact objects $D^c(\C A)$ is equivalent to $(\C T,\Sigma)$ as a suspended category. A suspended category is a category equipped with a self-equivalence, and two suspended categories $(\C T,\Sigma)$, $(\C T',\Sigma')$ are equivalent if there is an equivalence of categories $F\colon\C T\rightarrow\C T'$ such that $F\Sigma\cong\Sigma'F$. This is equivalent to the existence of a graded category equivalence $\C T_\Sigma\simeq\C T'_{\Sigma'}$.
\end{definition}

The aim of this paper is to show that, under mild conditions, $\etc{\C T,\Sigma}$ is either empty or a singleton. In this section, we give alternative descriptions of this set which will help in proving that.

Let $\morita$ be the category of DG-categories endowed with Tabuada's Morita model structure \cite{tabuada_invariants_2005, tabuada_addendum_2006, tabuada_corrections_2007}. Let $\morita(\C T,\Sigma)$ be the full subcategory of DG-categories $\C A$ with $D^c(\C A)$ equivalent to $(\C T,\Sigma)$. The set $\etc{\C T,\Sigma}$ has been implicitly defined as the set of connected components of the classification space $\abs{\morita(\C T,\Sigma)}$ in the sense of \cite{dwyer_classification_1984}. This space is the nerve of the category of Morita weak equivalences between DG-categories $\C A$ with $D^c(\C A)$ equivalent to $(\C T,\Sigma)$. This space contains much more information than its mere set of connected components. In a sense, it knows about the homotopy symmetries of each DG-category in the Morita model structure \cite[Corollary A.0.4]{toen_homotopical_2008}. 

Let $\equivalences$ be the category of DG-categories endowed with the model structure in \cite{tabuada_structure_2005}, whose weak equivalences are quasi-equivalences. Let us denote by $\equivalences(\C T,\Sigma)$ be the full subcategory of pre-triangulated DG-categories $\C A$, in the sense of \cite{bondal_enhanced_1991}, such that $H^0(\C A)$ is equivalent to $(\C T,\Sigma)$ as a suspended category. A DG-category $\C A$ is \emph{pre-triangulated} if the canonical inclusion functor $H^0(\C A)\hookrightarrow D(\C A)$ identifies $H^0(\C A)$ with a triangulated subcategory of the target. A DG-category $\C A$ is fibrant in $\morita$ precisely if it is pre-triangulated and $H^0(\C A)$ is idempotent complete. This is equivalent to saying that the previous canonical inclusion induces an equivalence $H^0(\C A)\simeq D^c(\C A)$ with the full subcategory of compact objects in $D(\C A)$. In particular, objects in $\equivalences(\C T,\Sigma)$ are Morita fibrant. Since $\morita$ is a left Bousfield localization of $\equivalences$, Morita equivalences between Morita fibrant DG-categories are the same as quasi-equivalences.

\begin{proposition}\label{intermediate}
	There is a homotopy equivalence \[\abs{\equivalences(\C T,\Sigma)}\simeq\abs{\morita(\C T,\Sigma)}.\]
\end{proposition}

\begin{proof}
	We have an inclusion functor $j\colon \equivalences(\C T,\Sigma)\subset\morita(\C T,\Sigma)$. A fibrant replacement functor $R$ on $\morita$ (co)restricts to $R\colon \morita(\C T,\Sigma)\rightarrow\equivalences(\C T,\Sigma)$. Both functors $j$ and $R$ preserve weak equivalences. Moreover, both composites $jR$ and $Rj$ are equipped with natural weak equivalences to the corresponding identity functor. Hence, we are done.
\end{proof}

Let $\Lambda$ be the endomorphism algebra of a basic additive generator of $\C T$, that we assume to be Frobenius as a necessary condition for $\etc{\C T,\Sigma}$ to be non-empty,  $\sigma\colon\Lambda\rightarrow\Lambda$ an automorphism induced by $\Sigma$, and  $\equivalences(\Lambda(\sigma))$ the full subcategory of $\equivalences$ spanned by the DG-categories $\C B$ satisfying the following properties:
\begin{enumerate}
	\item The category $H^0(\C B)$ has a unique object up to isomorphism. 
	\item The graded endomorphism algebra of some (and hence any) object in $H^*(\C B)$ is isomorphic to $\Lambda(\sigma)$. 
	\item The functor $D^c(\C B)\rightarrow\modules{\Lambda}$ defined by evaluating at an object  induces an equivalence onto the full subcategory of finitely presented projective $\Lambda$-modules. 
\end{enumerate}
Clearly, (3) is equivalent to the following condition:
\begin{enumerate}
	\item[$(3')$] The functor $H^0(\C B)\hookrightarrow D^c(\C B)$ induces an equivalence from the completion of the source by direct sums and idempotents.
\end{enumerate}
This completion, which will also be used below, consists of first formally adding all finite direct sums of objects and then all retracts of idempotents. 
By the Morita invariance of the derived category, $(3')$ is equivalent to:
\begin{enumerate}
	\item[$(3'')$] The completion of $\C B$ by direct sums and idempotents in $H^0(\C B)$ is pre-triangulated.
\end{enumerate}
Moreover, after (2), (3) is also equivalent to 
\begin{enumerate}
	\item[$(3''')$] The functor $D^c(\C B)_\Sigma\rightarrow\modules{\Lambda(\sigma)}$ defined by evaluating at an object induces an equivalence onto the full graded subcategory of finitely presented projective $\Lambda(\sigma)$-modules. 
\end{enumerate}
This is because the source and target of the functor in $(3''')$ are weakly stable, and its degree $0$ part is the functor in (3). 

\begin{definition}\label{ump}
	A \emph{minimal $A$-infinity algebra structure} $(A,m_3,\dots,m_n,\dots)$ on a graded algebra $A$ consists of Hochschild cochains $m_n\in\hc{n,2-n}{A,A}$, $n\geq 3$, satisfying some well-known equations \cite{keller_introduction_2001}. These equations imply that $m_3$ is a cocycle, whose cohomology class is called \emph{universal Massey product},
	\[\{m_3\}\in \hh{3,-1}{A,A}.\]
\end{definition}

\begin{remark}\label{ump2}
	Actually, the universal Massey product is well defined for any minimal $A_3$-algebra structure on $A$ which admits an $A_4$-extension. Using Kadeishvili's theorem \cite{kadeishvili_theory_1980, lefevre-hasegawa_sur_2003}, we can also define the universal Massey product of a DG-algebra as the universal Massey product of any minimal model. The universal Massey product has been previously studied in e.g.~\cite{kadeishvili2,benson}.
\end{remark}

By \cite[Corollary 6.2]{muro_first_2019} and the Morita invariance of Hochschild cohomology (i.e.~it does not change if we complete by direct sums and idempotents), we can also replace $(3)$ in the previous list with:
\begin{enumerate}
	\item[$(3'''')$] The universal Massey product of $\C B$ is an edge unit.
\end{enumerate}

\begin{proposition}
	There is a homotopy equivalence \[\abs{\equivalences(\Lambda(\sigma))}\simeq\abs{\equivalences(\C T,\Sigma)}.\]
\end{proposition}

\begin{proof}
	We define functors
	\[\equivalences(\Lambda(\sigma))\mathop{\rightleftarrows}\limits^R_{j}\equivalences(\C T,\Sigma)\]
	in the following way. The functor $j$ takes $\C A$ to the full sub-DG-category $j(\C A)$ spanned by the objects which become basic additive generators in $H^0(\C A)$. Therefore $j(\C A)$ satisfies $(1)$, $(2)$, and $(3'')$, since the completion of $j(\C A)$ mentioned therein is quasi-equivalent to $\C A$. In particular, the full inclusion $j(\C A)\subset\C A$ is a Morita equivalence. The functor $R$ is the (co)restriction of a fibrant replacement functor in $\morita$. 
	
	In order for $R$ to be well defined we must check that, for each $\C B$ in the source, $D^c(R(\C B))$ is equivalent to $(\C T,\Sigma)$. The natural cofibration $\C B\rightarrowtail R(\C B)$ is a Morita equivalence, so it induces a suspended equivalence $D^c(\C B)\simeq D^c(R(\C B))$, and  $D^c(\C B)$ is equivalent to $(\C T,\Sigma)$ by $(3''')$. 
	
	Both $j$ and $R$ preserve quasi-equivalences. There is an obvious natural quasi-equivalence from the identity in $\equivalences(\Lambda(\sigma))$ to $jR$. Hence, both functors induce homotopic maps on the classification space. 
	
	We will finish the proof as soon as we show that the functor $Rj$ induces a map homotopic to the identity in the classification space of $\equivalences(\C T,\Sigma)$. For this, we choose a functorial factorization in $\morita$, which functorially sends each DG-functor $f\colon\C A\rightarrow\C B$ to a factorization
	\[\C A\stackrel{\sim}{\rightarrowtail}R'(f)\twoheadrightarrow\C B\]
	of $f$ consisting of a trivial cofibration followed by a fibration. Therefore, here, and also below in this proof, $\sim$ stands for weak equivalence in the Morita model structure. We can suppose that $R(\C A)=R'(\C A\rightarrow e)$, where $e$ is the terminal DG-category. Given $\C A$ in $\equivalences(\C T,\Sigma)$, we can apply the functorial factorization to the commutative square
	\begin{center}
		\begin{tikzcd}
		j(\C A)\ar[r,"f", "\sim"']\ar[d, equal]&\C A\ar[d, two heads]\\
		j(\C A)\ar[r]&e
		\end{tikzcd}
	\end{center}
	where the top arrow is the inclusion which, as we pointed out above, is a Morita equivalence. This yields
	\begin{center}
		\begin{tikzcd}
		j(\C A)\ar[r, tail, "\sim"]\ar[d, equal]&R'(f)\ar[r, two heads]\ar[d]&\C A\ar[d, two heads]\\
		j(\C A)\ar[r, tail, "\sim"]&Rj(\C A)\ar[r, two heads]&e
		\end{tikzcd}
	\end{center}	
	Here, the top right horizontal arrow and the middle vertical arrow are Morita equivalences between fibrant objects in $\morita$. Here we use the 2-out-of-3 property and the fact that $f$ is a Morita equivalence. Therefore, those two maps are quasi-equivalences. This shows that the maps induced by the functors $Rj$ and $\C A\mapsto R'(j(\C A)\rightarrow\C A)$ are homotopic to the identity map in $\abs{\equivalences(\C T,\Sigma)}$.
\end{proof}

Let $\dgalgebras$ be the usual model category of (unital) DG-algebras \cite[\S5]{schwede_algebras_2000}. Weak equivalences are quasi-isomorphisms and fibrations are surjections. Let $\dgalgebras(\Lambda(\sigma))$ the full subcategory of DG-algebras $A$ with $H^*(A)\cong\Lambda(\sigma)$ whose universal Massey product is an edge unit.

\begin{proposition}
	There is a bijection
	\[\pi_0\abs{\dgalgebras(\Lambda(\sigma))}\simeq\pi_0\abs{\equivalences(\Lambda(\sigma))}.\]
\end{proposition}

\begin{proof}
	We can regard any DG-algebra as a DG-category with only one object. This defines an inclusion $\dgalgebras\subset \equivalences$ preserving weak equivalences. Using the characterization of $\equivalences(\Lambda(\sigma))$ in terms of universal Massey products, given by $(1)$, $(2)$, and $(3'''')$ above, we see that the previous inclusion (co)restricts to $\dgalgebras(\Lambda(\sigma))\subset \equivalences(\Lambda(\sigma))$. This is where the universal Massey product comes into play.
	
	Using \cite[Corollary A.0.5]{toen_homotopical_2008}, and arguing as in the proof of \cite[Proposition 2.3.3.5]{toen_homotopical_2008}, it is easy to see that the homotopy fiber of the map
	\[\abs{\dgalgebras(\Lambda(\sigma))}\longrightarrow\abs{\equivalences(\Lambda(\sigma))}\]
	induced by the previous inclusion at a DG-category $\C B$ in the target category is the mapping space
	\[\rmap_{\equivalences}(k,\C B),\]
	where $k$ is the ground field regarded as a DG-category with only one object with endomorphism algebra $k$. In general, $\pi_0$ of such mapping space is the set of isomorphism classes of objects in $H^0(\C B)$. In our case, this is a singleton, hence the statement follows.
\end{proof}

\begin{definition}\label{ets}
	An \emph{enhanced triangulated structure} on $(\C T,\Sigma)$ is a minimal $A$-infinity algebra structure on $\Lambda(\sigma)$ whose universal Massey product is an edge unit. Two enhanced triangulated structures are \emph{gauge equivalent} if there is an $A$-infinity morphism with identity linear part between them.
\end{definition}

The gauge equivalence relation is an honest equivalence relation by well-known properties of $A$-infinity morphisms. The quotient set will be denoted by
\[\ets{\C T,\Sigma}.\]

Definition \ref{ets} is explained by the characterization of pre-triangulated DG- and $A$-infinity categories, in the sense of Bondal-Kapranov \cite{bondal_enhanced_1991}, in terms of edge units, used above, compare \cite[Corollary 6.2]{muro_first_2019}.

We now relate the sets $\etc{\C T,\Sigma}$ and $\ets{\C T,\Sigma}$. 

\begin{remark}\label{action}
	Notice that the automorphism group $\aut(\Lambda(\sigma))$ of the graded algebra $\Lambda(\sigma)$ acts by conjugation on the right of the set of minimal $A$-infinity algebra structures $(\Lambda(\sigma),m_3,\dots,m_n,\dots)$ on $\Lambda(\sigma)$. More precisely, given $g\in \aut(\Lambda(\sigma))$,  \[(\Lambda(\sigma),m_3,\dots,m_n,\dots)^g=(\Lambda(\sigma),g^{-1}m_3g^{\otimes^3},\dots,g^{-1}m_ng^{\otimes^n},\dots).\] This action passes to $\ets{\C T,\Sigma}$. 
\end{remark}

\begin{theorem}\label{gauge}
	There is a bijection \[\ets{\C T,\Sigma}/\operatorname{Aut}(\Lambda(\sigma))\cong\etc{\C T,\Sigma}.\]
\end{theorem}

\begin{proof}
	Let $\chain$ be the category of chain complexes with the projective model structure. The forgetful functor $\dgalgebras\rightarrow\chain$ induces a map
	\[\abs{\dgalgebras}\longrightarrow\abs{\chain}.\]
	Forgetting the product, we can regard $\Lambda(\sigma)$ as a fibrant-cofibrant object in $\chain$ with trivial differential. By \cite[Theorem 4.6]{muro_moduli_2014}, \cite[Theorem 1.2 and Remark 6.5]{muro_homotopy_2016}, and \cite[Corollary 2.3]{muro_cylinders_2016}, the set of connected components $\pi_0F'$ of the homotopy fiber $F'$ of the previous map at $\Lambda(\sigma)$ is a quotient of the set of minimal $A$-infinity algebra structures $(\Lambda(\sigma), m_2',m_3,\dots,m_n,\dots)$ with underlying graded vector space $\Lambda(\sigma)$. Here, $m_2'$ endows $\Lambda(\sigma)$ with a unital graded associative algebra structure which may be different from the given one $m_2$. The equivalence relation is given by the existence of an $A$-infinity morphism with identity linear part, as in the definition of $\ets{\C T,\Sigma}$. 
	
	Let $\dgalgebras(\Lambda(\sigma))'$ be the full subcategory of $\dgalgebras$ spanned by all the DG-algebras $A$ with such that $H^*(A)$ is isomorphic to $\Lambda(\sigma)$ as graded vector spaces. The set $\pi_0\abs{\dgalgebras(\Lambda(\sigma))'}$ is by definition the image of the map $\pi_0F'\rightarrow\pi_0\abs{\dgalgebras}$, hence it is the quotient of $\pi_0F'$ by the action of $\pi_1(\abs{\chain},\Lambda(\sigma))$. This group is the automorphism group of $\Lambda(\sigma)$ as a graded vector space. The action on $\pi_0F'$, can be described as follows. The automorphism group of $\Lambda(\sigma)$ acts on the right of the previous set of $A$-infinity algebra structures $(\Lambda(\sigma), m_2',m_3,\dots,m_n,\dots)$ by conjugation, and this induces the action on the quotient set $\pi_0F'$. 
	
	The category $\dgalgebras(\Lambda(\sigma))$ is the full subcategory of $\dgalgebras(\Lambda(\sigma))'$ spanned by the objects $A$ with $H^*(A)\cong\Lambda(\sigma)$ as graded algebras whose universal Massey product is an edge unit. Let $F\subset F'$ be the full subspace spanned by the connected components such that the graded algebra $(\Lambda(\sigma),m_2')$ is isomorphic to $(\Lambda(\sigma), m_2)$ and the universal Massey product $\{m_3\}$ is an edge unit. Note that all these conditions are preserved by gauge equivalence. Then, we have that $\pi_0\abs{\dgalgebras(\Lambda(\sigma))}\cong \pi_0F/\pi_1(\abs{\chain},\Lambda(\sigma))$. Any element in the quotient has a representative with $m_2=m_2'$ and $\pi_0F$ can be alternatively described as the quotient of those representatives by the subgroup of automorphisms of the graded vector space $\Lambda(\sigma)$ which fix $m_2$. This is precisely the automorphism group of $\Lambda(\sigma)$ as a graded algebra. This theorem now follows from the previous results in this section.
\end{proof}

\section{Enhanced triangulated structures and edge units}\label{vacolap}

In this section, the ground field $k$ is required to be perfect. Consider a finite category $\C T$ equipped with an automorphism $\Sigma\colon\C T\rightarrow\C T$. Let $\Lambda$ be the endomorphism algebra of a basic additive generator. We assume that $\Lambda$ is Frobenius, as per Freyd's necessary condition for the existence of triangulated structures. Moreover, let $\sigma\colon\Lambda\rightarrow\Lambda$ be an automorphism induced by $\Sigma$. Recall from Definitions \ref{ump} and \ref{ets} the notions of universal Massey product, enhanced triangulated structure, and gauge equivalence. The group $\aut(\Lambda(\sigma))$ acts on the right of the Hochschild cohomology of $\Lambda(\sigma)$ by conjugation, more precisely, given $g\in \aut(\Lambda(\sigma))$ and $x\in \hh{\star,\ast}{\Lambda(\sigma),\Lambda(\sigma)}$, $x^g=g^*(g^{-1})_*(x)$. This is action is compatible with the Gerstenhaber algebra structure.

\begin{theorem}\label{classification}
	Universal Massey products define a bijection between $\ets{\C T,\Sigma}$ and the set of edge units $u\in\hh{3,-1}{\Lambda(\sigma),\Lambda(\sigma)}$ satisfying $\gsquare(u)=0$. This bijection is $\aut(\Lambda(\sigma))$-equivariant.
\end{theorem}

\begin{proof}
	The map in the statement is clearly $\aut(\Lambda(\sigma))$-equivariant, see Remark \ref{action}.
	
	The set of gauge equivalence classes of arbitrary minimal $A$-infinity algebra structures is
	\[\pi_0\operatorname{Map}_{\operatorname{dgOp}}(\mathtt{A}_\infty,\mathtt{E}(\Lambda(\sigma))).\]
	Here $\mathtt{A}_\infty$ is the $A$-infinity operad, $\mathtt{E}(\Lambda(\sigma))$ is the endomorphism operad of the graded vector space $\Lambda(\sigma)$, and $\operatorname{Map}_{\operatorname{dgOp}}$ stands for the mapping space in the model category of differential graded operads. The universal Massey product is invariant by gauge equivalences, i.e.~if two minimal $A$-infinity algebra structures on the graded vector space $\Lambda(\sigma)$ are quasi-isomorphic by an $A$-infinity quasi-isomorphism with identity linear part then they have the same product and the same universal Massey product \cite[Lemme B.4.2]{lefevre-hasegawa_sur_2003}. Therefore the property of being an enhanced triangulated structure on $(\C T,\Sigma)$ is also invariant by gauge equivalences.
	
	The cofibrant DG-operad $\mathtt{A}_\infty$ is the union of a sequence of cofibrations $\mathtt{A}_n\subset\mathtt{A}_{n+1}$, $n\geq 2$, where $\mathtt{A}_n$ is the $A_n$ operad. In \cite{muro_enhanced_2020}, we extend the Bousfield--Kan spectral sequence of the tower of fibrations $\{X_n\}_{n\geq 0}$, with $X_n=\operatorname{Map}_{\operatorname{dgOp}}(\mathtt{A}_{n+2},\mathtt{E}(\Lambda(\sigma)))$ and bonding maps defined by restriction, for the computation of the homotopy groups of $X_\infty=\lim_nX_n=\operatorname{Map}_{\operatorname{dgOp}}(\mathtt{A}_\infty,\mathtt{E}(\Lambda(\sigma)))$. The Bousfield--Kan terms $E_r^{p,q}$ are only defined for $q\geq p\geq 0$ and $E_\infty^{p,q}$ contributes to $\pi_{q-p}X_\infty$. The differentials look like
	\[d_r\colon E_r^{p,q}\To E_r^{p+r,q+r-1}\]
	and they are only defined for $q>p\geq 0$, so the elements of the so-called \emph{fringed line}, i.e.~$E_r^{p,p}$, $p\geq 0$, are not defined as the homology of differentials. Actually, these terms are plain pointed sets for general towers of spaces, and moreover the terms $E_r^{p,p+1}$, $p\geq 0$, are possibly non-abelian groups.

	We extend the range of definition of the previous spectral sequence to all $q\in\mathbb{Z}$ for $p\geq 2r-3$, $p\geq 0$. Most new and old terms are endowed with $k$-vector space structures. The only pointed sets are $E_r^{p,p}$ for $0\leq p\leq r-2$, and a finite amount of remaining terms are abelian groups (non-abelian groups do not show up for our particular tower of spaces). Moreover, on each page $E_r$ we define differentials $d_r$ like above out of all terms except for $0\leq p=q\leq r-1$ (with the exception of $d_2$, which is defined on $E_2^{1,1}$), and the term $E_{r+1}^{p,q}$ is given by the homology of $d_r$ whenever $E_r^{p,q}$ has an incoming and an outgoing differential (the incoming differential is taken to be $d_r=0$ if $q>p<r$). The new terms do not contribute to the homotopy groups of $X_\infty$, but they help in the computation of the Bousfield--Kan terms, and they also contain obstructions, as we explain below. Moreover, we fully compute the $E_2$ terms and the differential $d_2$ of our extended spectral sequence, in particular in the Bousfield--Kan part, where this was not previously known either. 
	
	The spectral sequence is defined if a base point in $X_\infty$ is given. If we only have a base point $x_n\in X_n$, then the spectral sequence is defined up to the terms of page $\lfloor\frac{n+3}{2}\rfloor$, hence we call it \emph{truncated}. Moreover, there is an obstruction in $E^{n+1,n}_r$, $1\leq r\leq \frac{n+3}{2}$, which vanishes if and only if there exists a vertex $x_{n+1}\in X_{n+1}$ which has the same image in $X_{n-r+1}$ as $x_n$. 
	
	The universal Massey product $\{m_3\}$ of an enhanced triangulated structure is an edge unit and satisfies $\gsquare(\{m_3\})=0$ since this is the obstruction for $n=r=2$, see \cite[Proposition 6.7]{muro_enhanced_2020}. Tautologically, any element $u\in\hh{3,-1}{\Lambda(\sigma),\Lambda(\sigma)}$ is the universal Massey product of some minimal $A_3$-algebra structure $x_1\in X_1$ defined by a representing cocycle, that we fix, which extends to $A_4$. Let $x_2'\in X_2$ be an extension. If $\gsquare(u)=0$, the obstruction vanishes and there is some $A_5$-algebra structure $x_3'\in X_3$ which restricts to $x_1\in X_1$. 
	
	Since we have $x_3'\in X_3$, the truncated spectral sequence is defined up to the $E_3$ terms, and the obstructions living therein are also defined. We will prove below that, if $u$ is an edge unit, then $E_3^{p,q}=0$ for all $p\geq 2$. Hence, the possibly non-trivial part of the $E_3$ page looks like
	\begin{center}
		\setcounter{trunco}{6}
		\setcounter{rpage}{4}
		\renewcommand{\arriba}{-3}
		\renewcommand{\derecha}{-3}
		\renewcommand{\abajo}{2}				
		\begin{tikzpicture}
		
		
		\draw[step=\rescale,gray,very thin] (0,{-\abajo*\rescale+\margen}) grid ({(3*\rpage -3 +\derecha)*\rescale-\margen},{(3*\rpage -3 +\arriba)*\rescale-\margen});
		
		
		\filldraw[fill=red,draw=none,opacity=0.2] (0,0) -- ({(\trunco-\rpage-1)*\rescale},{(\trunco-\rpage-1)*\rescale}) -- 
		({(\trunco-\rpage-1)*\rescale},{(3*\rpage -3 +\arriba)*\rescale-\margen}) --
		(0,{(3*\rpage -3 +\arriba)*\rescale-\margen}) -- cycle;
		\draw[red!70,thick,opacity=1]  (0,0) -- ({(\trunco-\rpage-1)*\rescale},{(\trunco-\rpage-1)*\rescale})-- ({(\trunco-\rpage-1)*\rescale},{(3*\rpage -3 +\arriba)*\rescale-\margen});

		
		\foreach  \x in {2,...,3}
		\node[fill=red,draw=none,circle,inner sep=.5mm,opacity=1]   at ({(\x-2)*\rescale},{(\x-2)*\rescale}) {};
		
		
		
		\draw [->] (0,0)  -- ({(3*\rpage -3 + \derecha)*\rescale},0) node[anchor=north] {$\scriptstyle s$};
		
		
		\draw [->] (0,-{\abajo*\rescale}) -- (0,{(3*\rpage -3 +\arriba)*\rescale})node[anchor=east] {$\scriptstyle t$};
		\end{tikzpicture}
	\end{center}
	Let us show now that this implies that the map in the statement is bijective. 
	
	In order to check surjectivity, we prove that there is an $A$-infinity algebra structure $x_\infty \in X_\infty$ which restricts to $x_1\in X_1$. More precisely, we prove by induction that, given elements $x_i\in X_i$, $1\leq i\leq n-2$, and $x_n'\in X_n$, with $x_1$ the fixed element above, compatible by restriction, we can obtain a similar collection of elements $x_i\in X_i$, $1\leq i\leq n-1$, and $x_{n+1}'\in X_{n+1}$. Here, the only new elements are $x_{n-1}$ and $x_{n+1}'$, and we have forgot $x_n'$. The initial case is $n=3$, defined above. For each $n\geq 3$, it suffices to show the existence of some $x_{n+1}'$ with the same image in $X_{n-2}$ as $x_n'$, i.e.~$x_{n-2}$, since we can then take $x_{n-1}$ as the image of $x_{n+1}'$ in $X_{n-1}$. The obstruction to this lives in $E_3^{n+1,n}=0$, hence we are done with surjectivity.
	
	Let us check injectivity. For this, we fix $x_\infty\in X_\infty$ restricting to $x_1\in X_1$. In particular, the whole spectral sequence is defined. We want to prove that any other $x_\infty'\in X_\infty$ with universal Massey product $u$ lies in the same connected component as $x_\infty$. Since the universal Massey products agree, the restriction of $x_\infty'$ to $X_1$ lies in the same component as $x_1$. This implies that the restrictions of $x_\infty$ and $x_\infty'$ to $X_n$ lie in the same connected component for all $n\geq 1$, since $E_3^{nn}$ is a singleton for any $n\geq 2$ and	
	\[E_3^{n,n}=\ker[\pi_0X_n\rightarrow\pi_0X_{n-1}]\cap\im[\pi_0X_{n+2}\rightarrow\pi_0X_{n}].\]
	This does not directly imply that $x_\infty$ and $x_\infty'$ lie in the same component of $X_\infty$. For this, we need to know that $\lim^1_n\pi_1X_n=0$, see \cite[IX.3.1]{bousfield_homotopy_1972}. This follows from \cite[IX.5.4]{bousfield_homotopy_1972} since the vanishing on $E_3$ implies that $E_3^{p,q}=E_r^{p,q}$ for all $q-p\geq 1$ and $r\geq 3$.
	
	The rest of this proof is devoted to the proof of the vanishing properties of $E_3$ claimed above. Recall from \cite{muro_enhanced_2020} that the $E_2$ terms (those which are defined) are 
	\begin{equation*}
	E_2^{p,q}=\left\{\begin{array}{ll}
	\hh{p+2,-q}{\Lambda(\sigma),\Lambda(\sigma)},&p>0, q\in\mathbb Z;\\
	\hz{2,-q}{\Lambda(\sigma),\Lambda(\sigma)}, & p=0, q>0.
	\end{array}\right.
	\end{equation*}
	Here, $\hz{\star,*}{\Lambda(\sigma),\Lambda(\sigma)}\subset\hc{\star,*}{\Lambda(\sigma),\Lambda(\sigma)}$ denotes the Hochschild cocycles. There is a remaining $E_2$ term, namely $E_2^{00}$, which is the pointed set of graded algebra structures with the same underlying graded vector space as $\Lambda(\sigma)$, based at $\Lambda(\sigma)$.

	The second differential \[d_2\colon E_2^{p,q}\longrightarrow E_2^{p+2,q+1},\] is defined except for $(p,q)=(0,0)$. It is given, up to sign, by the Gerstenhaber bracket with the universal Massey product,
	\[d_2=\pm[ u ,-],\]
	except for $(p,q)=(0,1)$. For $p=0$ and $q>1$, we understand that we first project the Hochschild cocycles onto the Hochschild cohomology and then apply $\pm[ u ,-]$. For $(p,q)=(0,1)$, the differential is given, up to sign, by this projection composed with the map
	\begin{equation}\label{alpha}
	\begin{split}
	\alpha\colon \hh{2,-1}{\Lambda(\sigma),\Lambda(\sigma)}&\longrightarrow \hh{4,-2}{\Lambda(\sigma),\Lambda(\sigma)},\\
	x&\;\mapsto\; x^2+ [ u ,x].
	\end{split}
	\end{equation}
	The first quadratic summand vanishes if $\characteristic k\neq 2$, since $\hh{\star,*}{\Lambda(\sigma),\Lambda(\sigma)}$ is graded commutative and $x$ has odd total degree. If $\characteristic k=2$, the quadratic summand is not $k$-linear unless $k=\mathbb F_2$, hence this differential, unlike the rest, is not a $k$-vector space morphism, but a plain abelian group morphism.
	
	Excluding $E_{2}^{00}$, where $d_2$ is not defined, the second page of the truncated spectral sequence splits into two families of cochain complexes $C^{*}_{n}$ and $D^{*}_{n}$, $n\in\mathbb Z$, 
	\begin{center}
		\setcounter{trunco}{5}
		\setcounter{rpage}{2}
		\renewcommand{\arriba}{1}
		\renewcommand{\derecha}{2}
		\renewcommand{\rescale}{1.2}			
		\renewcommand{\abajo}{1}				
		\begin{tikzpicture}
		
		
		\draw[step=\rescale,gray,very thin] (0,{-\abajo*\rescale+\margen}) grid ({(3*\rpage -3 +\derecha)*\rescale-\margen},{(3*\rpage -3 +\arriba)*\rescale-\margen});
		
		
		\filldraw[fill=red,draw=none,opacity=0.2] (0,0) -- 
		({(3*\rpage -3 +\arriba)*\rescale-\margen},{(3*\rpage -3 +\arriba)*\rescale-\margen}) --
		(0,{(3*\rpage -3 +\arriba)*\rescale-\margen}) -- cycle;
		\draw[red!70,thick,opacity=1]  (0,0) -- ({(\trunco-\rpage-1)*\rescale},{(\trunco-\rpage-1)*\rescale})-- 
		({(3*\rpage -3 +\arriba)*\rescale-\margen},{(3*\rpage -3 +\arriba)*\rescale-\margen});

		
		\foreach  \x in {2,...,\rpage}
		\node[fill=red,draw=none,circle,inner sep=.5mm,opacity=1]   at ({(\x-2)*\rescale},{(\x-2)*\rescale}) {};
		
		
		\filldraw[fill=blue,draw=none,opacity=0.2] ({(2*\rpage-3)*\rescale},{-\abajo*\rescale+\margen}) -- ({(3*\rpage -3 +\derecha)*\rescale-\margen},{-\abajo*\rescale+\margen}) --
		({(3*\rpage -3 +\derecha)*\rescale-\margen},{(3*\rpage -3 +\arriba)*\rescale-\margen}) -- ({(2*\rpage-3)*\rescale},{(3*\rpage -3 +\arriba)*\rescale-\margen}) -- cycle;
		\draw[blue!70,thick]  ({(2*\rpage-3)*\rescale},{-\abajo*\rescale+\margen}) -- ({(2*\rpage-3)*\rescale},{(3*\rpage -3 +\arriba)*\rescale-\margen});
		
		
		\draw [->] (0,0)  -- ({(3*\rpage -3 + \derecha)*\rescale},0);
		
		
		\draw [->] (0,-{\abajo*\rescale}) -- (0,{(3*\rpage -3 +\arriba)*\rescale});
		
		
		\foreach  \x in {-1,...,0}
		\draw [very thick] ({2*\rescale},{(\x+1)*\rescale}) -- ({(3*\rpage -3 +\derecha)*\rescale-\margen},{((3*\rpage -3 +\derecha+2*\x)*\rescale-\margen)*(1/2)}) node [right] {$\scriptstyle C^{*}_{\x}$};
		
		
		\draw [very thick] (0,{\rescale}) -- ({(3*\rpage -3 +\derecha)*\rescale-\margen},{((3*\rpage -3 +\derecha+2*1)*\rescale-\margen)*(1/2)}) node [right] {$\scriptstyle C^{*}_{1}$};
		
		\foreach  \x in {2,...,3}
		\draw [very thick] (0,{(\x)*\rescale}) -- ({(3*\rpage -3-2*(\x-2)+\arriba)*\rescale-\margen},{(3*\rpage -3 +\arriba)*\rescale-\margen}) node [above] {$\scriptstyle C^{*}_{\x}$};

		\foreach  \x in {-3,...,-2}
		\draw [very thick] ({-2*(\x+1)*\rescale+2*\margen},{-\abajo*\rescale+\margen}) -- ({(3*\rpage -3 +\derecha)*\rescale-\margen},{((3*\rpage -3 +\derecha+2*(\x))*\rescale-\margen)*(1/2)}) node [right] {$\scriptstyle C^{*}_{\x}$};
		
		\foreach  \x in {0,...,1}
		\draw [very thick, dashed] ({\rescale},{(\x)*\rescale}) -- ({(3*\rpage -3 +\derecha)*\rescale-\margen},{((3*\rpage -3 +\derecha+2*\x-1)*\rescale-\margen)*(1/2)})  node [right] {$\scriptstyle D^{*}_{\x}$};
		
		
		\foreach  \x in {-1,...,-2}
		\draw [very thick, dashed] ({-2*(\x+.5)*\rescale+2*\margen},{-\abajo*\rescale+\margen}) -- ({(3*\rpage -3 +\derecha)*\rescale-\margen},{((3*\rpage -3 +\derecha+2*(\x-.5))*\rescale-\margen)*(1/2)}) node [right] {$\scriptstyle D^{*}_{\x}$};
		
		\foreach  \x in {2,...,3}
		\draw [very thick, dashed] (\rescale,{(\x)*\rescale}) -- ({(3*\rpage -3-2*(\x-2.5)+\arriba)*\rescale-\margen},{(3*\rpage -3 +\arriba)*\rescale-\margen}) node [above] {$\scriptstyle D^{*}_{\x}$};
		
		\end{tikzpicture}
	\end{center}
	with the following descriptions,
	\begin{align*}
	C_{n}^{m}&=\left\{
	\begin{array}{ll}
	E_{2}^{2m,n+m},&m\geq 1,\text{ or } m=0\text{ and }n\geq 1;\\
	0,&\text{elsewhere};
	\end{array}
	\right.\\
	D_{n}^{m}&=\left\{
	\begin{array}{ll}
	E_{2}^{2m+1,n+m},&m\geq 0;\\
	0,&\text{elsewhere}.
	\end{array}
	\right.
	\end{align*}
	The differential is of course $d_{2}$ in all cases. The $E_3^{p,q}$ terms are defined for $q\geq p\geq 0$, and for $p\geq 3$ and $q\in\mathbb Z$. All of them are given by the homology of $d_2$, except for $E_3^{00}$ and $E_3^{11}$, i.e.
	\begin{align*}
	H^{m}C_{n}^{*}&=E_{3}^{2m,n+m}, &&m\geq 2,\text{ or } m\geq 0\text{ and }n\geq 1;\\
	H^{m}D_{n}^{*}&=E_{3}^{2m+1,n+m}, &&m\geq 1,\text{ or }  m\geq 0\text{ and }n\geq 2;
	\end{align*}
	
	In order to compute these groups, we consider new related families of cochain complexes, $\bar C^{*}_{n}$ and $\bar D^{*}_{n}$, $n\in\mathbb{Z}$, which agree with the former almost everywhere, and can be depicted as follows
	\begin{center}
		\setcounter{trunco}{6}
		\setcounter{rpage}{2}
		\renewcommand{\arriba}{1}
		\renewcommand{\derecha}{2}
		\renewcommand{\rescale}{1.2}			
		\renewcommand{\abajo}{1}				
		\begin{tikzpicture}
		
		
		\draw[step=\rescale,gray,very thin] (0,{-\abajo*\rescale+\margen}) grid ({(3*\rpage -3 +\derecha)*\rescale-\margen},{(3*\rpage -3 +\arriba)*\rescale-\margen});
		
		
		\draw [->] (0,0)  -- ({(3*\rpage -3 + \derecha)*\rescale},0);
		
		
		\draw [->] (0,-{\abajo*\rescale}) -- (0,{(3*\rpage -3 +\arriba)*\rescale});
		
		\foreach  \x in {0,...,1}
		\draw  [very thick] (0,{\x*\rescale}) -- ({(3*\rpage -3 +\derecha)*\rescale-\margen},{((3*\rpage -3 +\derecha+2*\x)*\rescale-\margen)*(1/2)}) node [right] {$\scriptstyle \bar C^{*}_{\x}$};
		
		\foreach  \x in {2,...,3}
		\draw [very thick] (0,{(\x)*\rescale}) -- ({(3*\rpage -3-2*(\x-2)+\arriba)*\rescale-\margen},{(3*\rpage -3 +\arriba)*\rescale-\margen}) node [above] {$\scriptstyle \bar C^{*}_{\x}$};
		
		\foreach  \x in {-3,...,-1}
		\draw  [very thick] ({-2*(\x+1)*\rescale+2*\margen},{-\abajo*\rescale+\margen}) -- ({(3*\rpage -3 +\derecha)*\rescale-\margen},{((3*\rpage -3 +\derecha+2*(\x))*\rescale-\margen)*(1/2)}) node [right] {$\scriptstyle \bar C^{*}_{\x}$};
		
		\foreach  \x in {0,...,1}
		\draw [very thick, dashed] (\rescale,{(\x)*\rescale}) -- ({(3*\rpage -3 +\derecha)*\rescale-\margen},{((3*\rpage -3 +\derecha+2*\x-1)*\rescale-\margen)*(1/2)})  node [right] {$\scriptstyle \bar D^{*}_{\x}$};
		
		\foreach  \x in {-2,...,-1}
		\draw [very thick, dashed] ({-2*(\x+.5)*\rescale+2*\margen},{-\abajo*\rescale+\margen}) -- ({(3*\rpage -3 +\derecha)*\rescale-\margen},{((3*\rpage -3 +\derecha+2*(\x-.5))*\rescale-\margen)*(1/2)}) node [right] {$\scriptstyle \bar D^{*}_{\x}$};
		
		\foreach  \x in {2,...,3}
		\draw [very thick, dashed] (\rescale,{(\x)*\rescale}) -- ({(3*\rpage -3-2*(\x-2.5)+\arriba)*\rescale-\margen},{(3*\rpage -3 +\arriba)*\rescale-\margen}) node [above] {$\scriptstyle \bar D^{*}_{\x}$};
		\end{tikzpicture}
	\end{center}
	Here, on any coordinate $(p,q)$, $p\geq 0$, $q\in\mathbb Z$, we place the Hochschild cohomology group $\hh{p+2,-q}{\Lambda(\sigma),\Lambda(\sigma)}$. Therefore, for $m\geq 0$ and $n\in\mathbb{Z}$,
	\begin{align*}
	\bar C_{n}^{m}&=\hh{2m+2,-n-m}{\Lambda(\sigma),\Lambda(\sigma)},&
	\bar D_{n}^{m}&=\hh{2m+3,-n-m}{\Lambda(\sigma),\Lambda(\sigma)},
	\end{align*}
	and $\bar C_{n}^{m}=\bar D_{n}^{m}=0$ for $m<0$. The differential is $[ u ,-]$ in all non-trivial cases. 
	
	The cup product with the universal Massey product $ u  \cdot-$ induces cochain maps
	\[f\colon \bar C_{n}^{*}\To \bar D_{n}^{*+1},\qquad
	g\colon \bar D_{n}^{*}\To \bar C_{n-1}^{*+2},
	\]
	for $n\in\mathbb{Z}$, 
	depicted below in red and blue, respectively,
	\begin{center}
		\setcounter{trunco}{6}
		\setcounter{rpage}{2}
		\renewcommand{\arriba}{1}
		\renewcommand{\derecha}{2}
		\renewcommand{\rescale}{1.2}			
		\renewcommand{\abajo}{1}				
		\begin{tikzpicture}
		
		
		\draw[step=\rescale,gray,very thin] (0,{-\abajo*\rescale+\margen}) grid ({(3*\rpage -3 +\derecha)*\rescale-\margen},{(3*\rpage -3 +\arriba)*\rescale-\margen});
		
		
		\draw [->] (0,0)  -- ({(3*\rpage -3 + \derecha)*\rescale},0);
		
		
		\draw [->] (0,-{\abajo*\rescale}) -- (0,{(3*\rpage -3 +\arriba)*\rescale});
		
		\foreach  \x in {0,...,1}
		\draw  [very thick] (0,{\x*\rescale}) -- ({(3*\rpage -3 +\derecha)*\rescale-\margen},{((3*\rpage -3 +\derecha+2*\x)*\rescale-\margen)*(1/2)}) node [right] {$\scriptstyle \bar C^{*}_{\x}$};
		
		\foreach  \x in {2,...,3}
		\draw [very thick] (0,{(\x)*\rescale}) -- ({(3*\rpage -3-2*(\x-2)+\arriba)*\rescale-\margen},{(3*\rpage -3 +\arriba)*\rescale-\margen}) node [above] {$\scriptstyle \bar C^{*}_{\x}$};
		
		\foreach  \x in {-3,...,-1}
		\draw  [very thick] ({-2*(\x+1)*\rescale+2*\margen},{-\abajo*\rescale+\margen}) -- ({(3*\rpage -3 +\derecha)*\rescale-\margen},{((3*\rpage -3 +\derecha+2*(\x))*\rescale-\margen)*(1/2)}) node [right] {$\scriptstyle \bar C^{*}_{\x}$};
		
		\foreach  \x in {0,...,1}
		\draw [very thick, dashed] (\rescale,{(\x)*\rescale}) -- ({(3*\rpage -3 +\derecha)*\rescale-\margen},{((3*\rpage -3 +\derecha+2*\x-1)*\rescale-\margen)*(1/2)})  node [right] {$\scriptstyle \bar D^{*}_{\x}$};
		
		\foreach  \x in {-2,...,-1}
		\draw [very thick, dashed] ({-2*(\x+.5)*\rescale+2*\margen},{-\abajo*\rescale+\margen}) -- ({(3*\rpage -3 +\derecha)*\rescale-\margen},{((3*\rpage -3 +\derecha+2*(\x-.5))*\rescale-\margen)*(1/2)}) node [right] {$\scriptstyle \bar D^{*}_{\x}$};
		
		\foreach  \x in {2,...,3}
		\draw [very thick, dashed] (\rescale,{(\x)*\rescale}) -- ({(3*\rpage -3-2*(\x-2.5)+\arriba)*\rescale-\margen},{(3*\rpage -3 +\arriba)*\rescale-\margen}) node [above] {$\scriptstyle \bar D^{*}_{\x}$};
		
		\foreach  \x in {-1,...,3}
		\foreach  \y in {0,...,1}
		\draw[red, ->, thick] ({2*\y*\rescale},{\x*\rescale}) -- ({(2*\y+3)*\rescale},{(\x+1)*\rescale});
		
		\foreach  \x in {-1,...,3}
		\draw[blue, ->, thick] (\rescale,{\x*\rescale}) -- ({4*\rescale},{(\x+1)*\rescale});

		\end{tikzpicture}
	\end{center}
	
	By Lemma \ref{non_singularity}, $f$ and $g$ are injective, the cokernel of $f$ is concentrated in degree $\ast=-1$, and the cokernel of $g$ is concentrated in degrees $\ast=-2, -1$. Therefore, the associated long exact sequences induce the following isomorphisms in cohomology for $m\geq 1$ and $n\in\mathbb{Z}$,
	\[f_{*}\colon H^{m}\bar C^{*}_{n}\cong H^{m+1}\bar D^{*}_{n},\qquad
	g_{*}\colon H^{m}\bar D^{*}_{n}\cong H^{m+2}\bar C^{*}_{n}.
	\]
	By Proposition \ref{nulo} and Remark \ref{nulo2}, $f$ and $g$ are null-homotopic, so, for $m\geq 1$ and $n\in\mathbb{Z}$,
	\[H^{m}\bar C^{*}_{n}=0=H^{m}\bar D^{*}_{n}.\]
	
	Clearly, $D^{*}_{n}=\bar D^{*}_{n}$, hence
	\[E_3^{2m+1,n+m}=H^mD^{*}_{n}=0,\qquad m\geq 1, n\in\mathbb{Z}.\]
	For $n\leq 0$, $C^*_n$ is the naive truncation of $\bar C^*_n$ at $*\geq 1$, therefore
	\[E_3^{2m,n+m}=H^mC^{*}_{n}=H^m\bar C^{*}_{n}=0,\qquad m\geq 2\text{ and }n\leq 0.\]
	For $n\geq 2$, there is an obvious surjective cochain map $C^*_n\twoheadrightarrow\bar C^*_n$ which is the identity in $*>0$ and the natural projection of Hochschild cocycles onto Hochschild cochains in $*=0$, so
	\[E_3^{2m,n+m}=H^mC^{*}_{n}=H^m\bar C^{*}_{n}=0,\qquad m\geq 1\text{ and }n\geq 2.\]
	For $n=1$ and $\characteristic k\neq 2$, we also have a surjective map $C^*_1\twoheadrightarrow\bar C^*_1$ for the same reason as above, and
	\[E_3^{2m,m+1}=H^mC^{*}_{1}=H^m\bar C^{*}_{1}=0,\qquad m\geq 1.\]
	This completes the proof in case $\characteristic k\neq 2$. The problem in $\characteristic k = 2$ is that the differential $C_1^0\to C_1^1$ depends on the quadratic map \eqref{alpha} while $\bar C_1^0\to \bar C_1^1$ is just given by $[u,-]$, so there is no obvious map $C^*_1\twoheadrightarrow\bar C^*_1$.
	
	The last equation also holds for $\characteristic k= 2$, although it is more complicated to check. It suffices to construct a chain map $C^*_1\twoheadrightarrow\bar D^{*+1}_1$ given by $ u  \cdot-$ for $*>0$, which is an isomorphism by Lemma \ref{non_singularity}, and which is surjective for $*=0$, since then
	\[E_3^{2m,m+1}=H^mC^{*}_{1}=H^{m+1}\bar D^{*}_{1}=0,\qquad m\geq 1.\]
	we define $C_1^0\rightarrow D^1_1$ as the following composite
	\[\hz{2,-1}{\Lambda(\sigma),\Lambda(\sigma)}\twoheadrightarrow\hh{2,-1}{\Lambda(\sigma),\Lambda(\sigma)}\stackrel{\beta}{\longrightarrow}\hh{5,-2}{\Lambda(\sigma),\Lambda(\sigma)}\]
	where 
	\[\beta(x)= u  \cdot x+\{\delta\} \cdot x^2.\] 
	We must show that this $\beta$ is surjective and completes the definition of a chain map $C^*_1\twoheadrightarrow\bar D^{*+1}_1$. Let us start with the second property. It suffices to check that
	\[ u  \cdot\alpha(x)=[ u ,\beta(x)]\]
	for any $x\in \hh{2,-1}{\Lambda(\sigma),\Lambda(\sigma)}$, where $\alpha$ is the morphism in \eqref{alpha} which defines the differential $d_2\colon E_2^{0,1}\rightarrow E_2^{2,2}$. This follows from Proposition \ref{formulilla}, $\characteristic k=2$, and the laws of a Gerstenhaber algebra, since we know that $[ u , u  \cdot x]= u  \cdot [ u ,x]$, and 
	\begin{align*}
	[ u ,\{\delta\} \cdot x^2]={}&[ u ,\{\delta\}] \cdot x^2
	+\{\delta\} \cdot [ u ,x] \cdot x
	+\{\delta\} \cdot x \cdot [ u ,x]\\
	={}& u  \cdot x^2
	+\{\delta\} \cdot [ u ,x] \cdot x
	+\{\delta\} \cdot [ u ,x] \cdot x\\
	={}& u  \cdot x^2.
	\end{align*}
	We finish by showing the surjectivity of $\beta$. Actually, we will see that it is bijective. For this, we construct morphisms $\gamma$ and $\lambda$ fitting in the following commutative diagram, 
	\begin{center}
		\begin{tikzcd}
		\hh{2,-1}{\Lambda(\sigma),\Lambda(\sigma)}\arrow[d,"\beta"]
		\arrow[dd,bend left=80, " u ^3 \cdot-","\cong"']\\
		\hh{5,-2}{\Lambda(\sigma),\Lambda(\sigma)}\arrow[d,"\gamma"]
		\arrow[dd,bend right=80, " u ^6 \cdot-"',"\cong"]\\
		\hh{11,-4}{\Lambda(\sigma),\Lambda(\sigma)}\arrow[d,"\lambda"]\\
		\hh{23,-8}{\Lambda(\sigma),\Lambda(\sigma)}
		\end{tikzcd}
	\end{center}
	where the curved arrows are isomorphisms by Proposition \ref{non_singularity}. The morphisms $\gamma$ and $\lambda$ are defined as
	\begin{align*}
	\gamma(x)&= u ^2 \cdot x+\{\delta\} \cdot x^2,\\
	\lambda(x)&= u ^4 \cdot x+\{\delta\} \cdot x^2.
	\end{align*}
	Commutativity follows from
	\begin{align*}
	\gamma\beta(x)={}&\gamma( u  \cdot x+\{\delta\} \cdot x^2)\\
	={}& u ^3 \cdot x+ u ^2 \cdot \{\delta\} \cdot x^2+\{\delta\} \cdot( u  \cdot x+\{\delta\} \cdot x^2)^2\\
	={}& u ^3 \cdot x+ u ^2 \cdot \{\delta\} \cdot x^2+\{\delta\} \cdot u ^2 \cdot x^2\\
	={}& u ^3 \cdot x,\\
	\lambda\gamma(x)={}&\lambda( u ^2 \cdot x+\{\delta\} \cdot x^2)\\
	={}& u ^6 \cdot x+ u ^4 \cdot\{\delta\} \cdot x^2+\{\delta\} \cdot( u ^2 \cdot x+\{\delta\} \cdot x^2)^2\\
	={}& u ^6 \cdot x+ u ^4 \cdot\{\delta\} \cdot x^2+\{\delta\} \cdot u ^4 \cdot x^2\\
	={}& u ^6 \cdot x.
	\end{align*}
	Here we use Proposition \ref{vaa0}, the commutativity of the cup product, and that $\characteristic k=2$.
\end{proof}

The group $\aut(\Lambda(\sigma))$ also acts on the right of the graded algebra $\hh{\star,*}{\Lambda,\Lambda(\sigma)}$ by conjugation since $\Lambda$ is the degree $0$ part of $\Lambda(\sigma)$, so any automorphism of the latter (co)restricts to an automorphism of the former. The following corollary is a consequence of the previous theorem and Proposition \ref{bijection}.

\begin{corollary}\label{restricted_suffices}
	There is an $\aut(\Lambda(\sigma))$-equivariant bijection between $\ets{\C T,\Sigma}$ and the set of edge units in $\hh{3,-1}{\Lambda,\Lambda(\sigma)}$.
\end{corollary}

\begin{proposition}
	The set $\etc{\C T,\Sigma}$ is non-empty if and only if $\Omega^3(\Lambda)\cong {}_{\sigma^{-1}}\Lambda_1$ in $\modulesst{\Lambda^\env}$. Moreover, in that case it is a singleton.
\end{proposition}

\begin{proof}
	By Theorem \ref{gauge}, $\etc{\C T,\Sigma}$ is non-empty if and only if $\ets{\C T,\Sigma}$ is. Therefore, the first part of the statement follows from the previous corollary and Propositions \ref{units} and \ref{edge_characterization}. 
	
	By those previous results, in order to prove the second part, we must check that $\htate{3,-1}{\Lambda,\Lambda(\sigma)}^\times/\aut(\Lambda(\sigma))$ is a singleton. Let $Z(\Lambda)$ denote the center of $\Lambda$. There is a group morphism $g\colon \hh{0,0}{\Lambda,\Lambda(\sigma)}^\times=\hh{0,0}{\Lambda,\Lambda}^\times=Z(\Lambda)^\times\rightarrow\aut(\Lambda(\sigma))$ sending $x\in Z(\Lambda)^\times$ to the automorphism $g(x)\colon\Lambda(\sigma)\rightarrow \Lambda(\sigma)$ defined on each degree $n\in\mathbb{Z}$ by right multiplication by $x^n$. This morphism is obviously injective, so we can regard $\hh{0,0}{\Lambda,\Lambda(\sigma)}^\times$ as a subgroup of $\aut(\Lambda(\sigma))$. Since any $g(x)$ is the identity in degree $0$, the right action of $\hh{0,0}{\Lambda,\Lambda(\sigma)}^\times$ on $\htate{3,-1}{\Lambda,\Lambda(\sigma)}^\times$ is given by the comparison morphism  $\hh{0,0}{\Lambda,\Lambda(\sigma)}^\times\twoheadrightarrow \htate{0,0}{\Lambda,\Lambda(\sigma)}^\times$, which is surjective by Proposition \ref{surjective_units}, and left multiplication by the inverse. The quotient \[\htate{3,-1}{\Lambda,\Lambda(\sigma)}^\times/\htate{0,0}{\Lambda,\Lambda(\sigma)}^\times\] 
	of the `left multiplication by the inverse' action
	is a singleton since, given $x,y\in \htate{3,-1}{\Lambda,\Lambda(\sigma)}^\times$, $xy^{-1}\in \htate{0,0}{\Lambda,\Lambda(\sigma)}^\times$. Therefore, the quotient by the larger group $\htate{3,-1}{\Lambda,\Lambda(\sigma)}^\times/\aut(\Lambda(\sigma))$ is also a singleton.
\end{proof}

\begin{definition}
	We say that $\C T$ has an enhanced triangulated structure if $(\C T,\Sigma)$ does for some automorphism $\Sigma\colon\C T\rightarrow\C T$.
\end{definition}

Recall from \cite[Proposition 3.8]{bolla_isomorphisms_1984} that the \emph{Picard group} $\pic(\Lambda)$ of invertible $\Lambda$-bimodules is isomorphic to the outer automorphism group $\out(\Lambda)$ via $\out(\Lambda)\rightarrow\pic(\Lambda)\colon [\sigma]\mapsto [{}_{\sigma}\Lambda_1]$. Moreover, $\pic(\Lambda)$ is also the group of natural isomorphism classes of self-equivalences of $\modules{\Lambda}$, and $[\sigma]$ corresponds to the restriction of scalars $(\sigma^{-1})^*\colon\modules{\Lambda}\rightarrow\modules{\Lambda}$ along $\sigma^{-1}\colon \Lambda\rightarrow\Lambda$, which is naturally isomorphic to $-\otimes_\Lambda{}_{\sigma}\Lambda_1$. The advantage of the equivalence $(\sigma^{-1})^*$ is that it is an automorphism.

\begin{corollary}\label{main_theorem}
	The finite category $\C T\simeq\proj{\Lambda}$ has an enhanced triangulated structure if and only if $\Omega^3(\Lambda)$ is isomorphic in $\modulesst{\Lambda^\env}$ to an invertible $\Lambda$-bimodule. Up to natural isomorphism, the possible suspension functors are the restrictions of scalars $\sigma^*\colon\proj{\Lambda}\rightarrow\proj{\Lambda}$, where $\sigma$ runs over a set of representatives of elements in $\out(\Lambda)$ such that $\Omega^3(\Lambda)\cong{}_{\sigma^{-1}}\Lambda_1$ in $\modulesst{\Lambda^\env}$. Once the suspension functor $\Sigma$ is fixed, $\etc{\C T,\Sigma}$ is a singleton.
\end{corollary}

\begin{corollary}
	If $\Lambda$ is separable, then $\C T\simeq\proj{\Lambda}$ has an enhanced triangulated structure.
	Up to natural isomorphism, the possible suspension functors are the restrictions of scalars $\sigma^*\colon\proj{\Lambda}\rightarrow\proj{\Lambda}$, where $\sigma$ runs over a set of representatives of elements in $\out(\Lambda)$. Once the suspension functor $\Sigma$ is fixed, $\etc{\C T,\Sigma}$ is a singleton.
\end{corollary}

The case analyzed in this corollary is particularly simple, $\Lambda$ is semisimple, $\C T\simeq D^c(\Lambda(\sigma))$, and all exact triangles split here.

Recall that the algebra $\Lambda$ is \emph{connected} if it cannot be decomposed as a product  $\Lambda\cong\Lambda_{1}\times\Lambda_{2}$, with $\Lambda_{i}$ a non-trivial algebra, $i=1,2$.

\begin{proposition}\label{irreducible}
	Suppose $\Lambda$ is connected and not separable. Given any $\Lambda$-bimodule $M$, denote by $\Omega(M)$ the syzygy obtained ascovercover the kernel of a projective cover,
	\[\Omega(M)\hookrightarrow P\twoheadrightarrow M.\]
	Then, $\C T=\proj{\Lambda}$ has an enhanced triangulated structure if and only if $\Omega^3(\Lambda)$ is an invertible $\Lambda$-bimodule. In that case, the suspension functor is necessarily $\Sigma=-\otimes_{\Lambda}\Omega^{3}(\Lambda)^{-1}$, where $\Omega^{3}(\Lambda)^{-1}$ is an inverse of $\Omega^3(\Lambda)$ in $\pic(\Lambda)$. Moreover, $\etc{\C T,\Sigma}$ is a singleton.
\end{proposition}

\begin{proof}
	The $\Lambda$-bimodule $\Lambda$ is indecomposable, since $\Lambda$ is connected as an algebra, hence $\Omega^3(\Lambda)$ too, and also any invertible $\Lambda$-bimodule ${}_\sigma\Lambda_{1}$. Moreover, ${}_\sigma\Lambda_{1}$ cannot be projective. Otherwise, it would be a direct summand of $\Lambda\otimes\Lambda$, and then $\Lambda$ would be a direct summand of ${}_{\sigma^{-1}}\Lambda\otimes\Lambda\cong\Lambda\otimes\Lambda$, hence separable. The third syzygy $\Omega^3(\Lambda)$ is not projective either since that would also imply that $\Lambda$ would be a projective bimodule.
	
	Two $\Lambda$-bimodules $M$ and $N$ are isomorphic in $\modulesst{\Lambda^\env}$ if and only if there are projective $\Lambda$-bimodules $P$ and $Q$ such that $M\oplus P\cong N\oplus Q$ in $\modules{\Lambda^\env}$. Hence, $M=\Omega^3(\Lambda)$ is isomorphic to some $N={}_\sigma\Lambda_{1}$ in $\modulesst{\Lambda^\env}$ if and only if they are isomorphic in $\modules{\Lambda^\env}$, because each of them is the only non-projective indecomposable factor on each side of an isomorphism $M\oplus P\cong N\oplus Q$ as above. Now, this proposition follows from Corollary \ref{main_theorem}.
\end{proof}

Recall that the Nakayama algebra $N_m^n$ is the quotient of the path algebra of the oriented cycle of length $m\geq 1$ 
\begin{center}
	\begin{tikzcd}
		&m \arrow[r,"\alpha_m"]&1\arrow[rd,"\alpha_1"]&\\
		m-1\arrow[ru,"\alpha_{m-1}"]&&&2\arrow[d,"\alpha_2"]\\
		m-2\arrow[u,"\alpha_{m-2}"]
		\arrow[rrr,dotted,-,bend right=90]&&&3
	\end{tikzcd}
\end{center}
by the two-sided ideal generated by the paths of length $n+1$, $n\geq 1$. This basic and self-injective algebra (hence Frobenius) has dimension $m(n+1)$ over the ground field $k$ and it is connected but not separable, see \cite[\S3]{holm_hochschild_1998} and \cite[\S3]{erdmann_twisted_1999}.

\begin{proposition}\label{counter}
	If $k$ is algebraically closed, $\Lambda=N_m^n$, $m$ divides $n$, and $n>1$, then $\Omega^3(\Lambda)$ is not stably isomorphic to an invertible $\Lambda$-bimodule, hence $\proj{\Lambda}$ cannot be endowed with a(n enhanced) triangulated category structure.
\end{proposition}
	
\begin{proof}
	In this proof we compute syzygies by using projective covers. By \cite[\S3]{holm_hochschild_1998}, $\Omega^2(\Lambda)={}_{\sigma}\Lambda_1$ for some automorphism $\sigma$ of $\Lambda$. Since $\Lambda$ is connected and not separable, if $\Omega^3(\Lambda)$ were stably isomorphic to an invertible $\Lambda$-bimodule then it would actually be isomorphic to it, by the same argument as in the proof of Proposition \ref{irreducible} above. Recall from \cite[Proposition 3.8]{bolla_isomorphisms_1984} that any invertible $\Lambda$-bimodule is of the form ${}_\mu\Lambda_1$ for some automorphism $\mu$. If that happened, the minimal projective resolution of $\Lambda$ as a bimodule would produce a short exact sequence of bimodules
	\[{}_\mu\Lambda_1\hookrightarrow P\twoheadrightarrow {}_\sigma\Lambda_1\]
	with projective middle term. If we twisted by $\sigma^{-1}$ from the left we would obtain another short exact sequence of bimodules with projective middle term
	\[{}_{\sigma^{-1}\mu}\Lambda_1\hookrightarrow {}_{\sigma^{-1}}P_1\twoheadrightarrow \Lambda,\]
	hence 
	$\Omega(\Lambda)={}_{\sigma^{-1}\mu}\Lambda_1$. We know that $\dim_k{}_{\sigma^{-1}}P_1=\dim_kP=m(n+1)^2$, see \cite[\S3]{holm_hochschild_1998}, and $\dim_k{}_{\sigma^{-1}\mu}\Lambda_1=\dim_k\Lambda=m(n+1)$, so this can only happen if $n=1$.
	
	The final conclusion follows from (1) in our main theorem in the enhanced case and from \cite[Theorem 1.2]{hanihara_auslander_2018} in the non-enhanced case.
\end{proof}

For $n=1$, $\Lambda=N_m^1$ satisfies $\Omega(\Lambda)={}_{\sigma}\Lambda_1$ for some automorphism $\sigma$ of $\Lambda$, see the proof of \cite[4.2]{erdmann_twisted_1999}. Hence $\Omega^3(\Lambda)={}_{\sigma^3}\Lambda_1$, so in this case the Nakayama algebra $\Lambda$ does satisfy the assumptions of Corollary \ref{main_theorem}.


\providecommand{\bysame}{\leavevmode\hbox to3em{\hrulefill}\thinspace}
\providecommand{\MR}{\relax\ifhmode\unskip\space\fi MR }
\providecommand{\MRhref}[2]{%
  \href{http://www.ams.org/mathscinet-getitem?mr=#1}{#2}
}
\providecommand{\href}[2]{#2}

\end{document}